\documentclass{ws-m3as}
\usepackage{cite}
\usepackage{amssymb,mathrsfs,graphicx,extpfeil}
\usepackage{epsfig}
\usepackage{indentfirst, latexsym, amssymb, enumerate,amsmath,graphicx}
\usepackage{float}
\usepackage{colortbl}
\usepackage{epsfig,subfigure}
\usepackage{amsmath}
\allowdisplaybreaks[4]

\newcommand{\bbr}{\mathbb R}

\newcommand{\bbn} {\mathbb N}

\newcommand{\N}{\mathcal N}
\newcommand{\G}{\mathcal G}
\newcommand{\V}{\mathcal V}

\newcommand{\D}{\mathcal D}

\newcommand{\E}{\mathcal E}

\begin{document}

\markboth{J.-G. Dong, S.-Y. Ha and D. Kim}{Discrete and continuous thermomechanical Cucker-Smale models on general digraph}

%
\catchline{}{}{}{}{}
%

\title{Emergent behaviors of continuous and discrete thermomechanical Cucker-Smale models on general digraphs}

\author{Jiu-Gang Dong}
\address{Department of Mathematics, Harbin Institute of
Technology, Harbin 150001, China \\
jgdong@hit.edu.cn}

\author{Seung-Yeal Ha}
\address{Department of Mathematical Sciences and Research Institute of Mathematics \newline Seoul National University, Seoul 08826, and \newline
Korea Institute for Advanced Study, Hoegiro 85, Seoul, 02455,
Korea (Republic of) \\
syha@snu.ac.kr}

\author{Doheon Kim}
\address{Department of Mathematical Sciences, Seoul National University, \newline Seoul 08826, Korea (Republic of), dohun930728@snu.ac.kr}

\maketitle


\begin{abstract}
We present emergent dynamics of continuous and discrete thermomechanical Cucker-Smale(TCS) models equipped with temperature as an extra observable on  general digraph. In previous literature, the emergent behaviors of the TCS models were mainly studied on a complete graph, or symmetric connected graphs. Under this symmetric  setting, the total momentum is a conserved quantity. This determines the asymptotic velocity and temperature a priori using the initial data only. Moreover, this conservation law plays a crucial role in the flocking analysis based on the elementary $\ell_2$ energy estimates. In this paper, we consider a more general connection topology which is registered by a general digraph, and the weights between particles are given to be inversely proportional to the metric distance between them. Due to this possible symmetry breaking in communication, the total momentum is not a conserved quantity, and this lack of conservation law makes the asymptotic velocity and temperature depend on the whole history of solutions. To circumvent this lack of conservation laws, we instead employ some tools from matrix theory on the scrambling matrices and some detailed analysis on the state-transition matrices. We present two sufficient frameworks for the emergence of mono-cluster flockings on a digraph for the continuous and discrete models. Our sufficient frameworks are given in terms of system parameters and initial data. 
\end{abstract}

\keywords{Digraph, emergence, energy estimate, scrambling matrices, state-transition matrices, thermomechanical Cucker-Smale particles}

\ccode{AMS Subject Classification: 39A11, 39A12, 34D05, 68M10}

\section{Introduction} \label{sec:1}
\setcounter{equation}{0}
Collective behaviors of many-particle systems are ubiquitous in our nature, e.g., flocking of birds, flashing of fireflies, swarming of fishes and herding of sheep, etc \cite{BC, B-H, L-P, T-T, VZ12}. Among the diverse collective behaviors, our concern lies on the flocking which denotes some concentration phenomenon in velocity, in which particles move with the same velocity asymptotically only using the simple environment information and basic rules. Motivated by the seminal contributions \cite{Rey87, VKJS95} on the flocking modeling by Reynolds and Vicsek et al, several mechanical models were proposed in literature in diverse disciplines such as applied mathematics, control theory of multi-agent system and statistical physics.  Among them, our main interest lies on the flocking models proposed by Cucker and Smale \cite{CS07}. In this paper, we are mainly interested in the thermodynamic Cucker-Smale model which generalizes the classical Cucker-Smale model by adding internal temperature variables. More precisely, let $x_i, v_i, \theta_i$ be the position, velocity and temperature of the $i$-th Cucker-Smale particles. Then, dynamics of these macroscopic observables is governed by the Cauchy problem to the continuous thermodynamic Cucker-Smale model introduced in \cite{H-R}:
\begin{align}
\begin{aligned}\label{A-0}
&\frac{dx_i}{dt}=v_i,\quad t>0, \quad i=1,2,\cdots, N,\\
&\frac{dv_i}{dt}=\frac{1}{N}\sum_{j=1}^N \chi_{ij}\phi(\|x_i-x_j \|)\Big(\frac{v_j}{\theta_j}-\frac{v_i}{\theta_i}\Big),\\
&\frac{d\theta_i}{dt}=\frac{1}{N}\sum_{j=1}^N \chi_{ij}\zeta(\|x_i-x_j\|)\Big(\frac{1}{\theta_i}-\frac{1}{\theta_j}\Big),\\
\end{aligned}
\end{align}
subject to the initial data:
\[(x_i(0),v_i(0),\theta_i(0))=(x_i^{in},v_i^{in},\theta_i^{in}),\quad i=1,2,\cdots, N\]
where $\|\cdot\|$ denotes the standard $\ell^2$-norm in $\bbr^d$. Here network topology $(\chi_{ij})$ is given as follows:
\[ \chi_{ij} = \begin{cases}
1, \quad & \mbox{if $j$ transmits information to $i$}, \cr
0, \quad & \mbox{otherwise,}
\end{cases}
\] Here, we assume that the directed graph corresponding to the network topology $(\chi_{ij})$ has at least one spanning tree. And we also assume $\chi_{ii}=1$ for $i=1,\cdots,N$, for some technical reason. The communication weights $\phi, \zeta: \mathbb{R}_+\cup\{0\}\to \mathbb{R}_+$ appearing in the R.H.S. of \eqref{A-0} are assumed to be bounded, Lipschitz continuous positive non-increasing functions defined on the nonnegative real numbers:
\begin{align*}
\begin{aligned} 
& 0 < \phi(r) \leq \phi(0) =: \kappa_1, \quad 0 < \zeta(r) \leq \zeta(0) =: \kappa_2, \quad r \geq 0, \\
& (\phi(r)-\phi(s))(r-s)\le0,\quad (\zeta(r)-\zeta(s))(r-s)\le 0,\quad r,s\ge0.
\end{aligned}
\end{align*}
Note that if we choose all initial temperatures $\theta^{in}_i$ to have the same value, say unity, then it is easy to see that $\theta_i(t) \equiv 1,~i=1,\cdots, N$  and system \eqref{A-0} reduces to the Cucker-Smale (CS) model \cite{CS07} on a digraph:
\begin{align*}
\begin{aligned}
&\frac{dx_i}{dt}=v_i,\quad t>0, \quad i=1,2,\cdots, N,\\
&\frac{dv_i}{dt}=\frac{1}{N}\sum_{j=1}^N \chi_{ij}\phi(\|x_i-x_j \|) (v_j - v_i).
\end{aligned}
\end{align*}
The CS model has received lots of attention in literature from different perspectives, i.e., global and local flocking \cite{ C-H-H-J-K1, CS07,CS071, H-K-Z,  H-Liu, H-T}, collision avoiding \cite{A-C-H-L, CD10}, time-delay effect \cite{CH17, C-L, D-H-K, E-H-S, LW14, PV17}, hierarchical and rooted leadership, general digraph \cite{DQ17, L-H, L-X, She}, application to flight navigation \cite{P-E-G}, noisy effects \cite{A-H, C-M,  D-F-T, H-L-L}, mean-field limit \cite{B-C-C, H-Liu, H-K-Z}, kinetic and hydrodynamic description \cite{C-C-R, C-F-R-T, F-H-T, H-T, P-S}, and variants of C-S model \cite{M-T2, P-K-H}, etc. (see recent survey papers \cite{C-H-L, M-T1} for details). Compared to the above vast literature on the C-S model, there were very few works for the TCS model, e.g., global flocking \cite{H-K-R, H-R}, flocking effect by singular communication \cite{C-H-K}, hydrodynamic TCS model \cite{H-K-M-R-Z1, H-K-M-R-Z2}. \newline

Next, we consider the discrete analogue of \eqref{A-0}:
\begin{align}
\begin{aligned}\label{A-1}
&x_i[t+1]=x_i[t]+h v_i[t],\quad t\in\bbn\cup\{0\}, \quad i=1,2,\cdots, N,\\
&v_i[t+1]=v_i[t]+\frac{h}{N}\sum_{j=1}^N \chi_{ij}\phi(\|x_i[t]-x_j[t]\|)\Big(\frac{v_j[t]}{\theta_j[t]}-\frac{v_i[t]}{\theta_i[t]}\Big),\\
&\theta_i[t+1]=\theta_i[t]+\frac{h }{N}\sum_{j=1}^N \chi_{ij}\zeta(\|x_i[t]-x_j[t]\|)\Big(\frac{1}{\theta_i[t]}-\frac{1}{\theta_j[t]}\Big),\\
\end{aligned}
\end{align}
with the initial data
\[(x_i[0],v_i[0],\theta_i[0])=(x_i^{in},v_i^{in},\theta_i^{in}),\quad i=1,2,\cdots, N.\]
Here $h>0$ denotes the time-step. \newline

Before, we discuss our main results obtained in this paper, we recall the concept of mono-cluster flocking as follows. We first set $X, V$ and $\Theta$:
\begin{align*}
\begin{aligned}
& X := (x_1, \cdots, x_N), \quad V := (v_1, \cdots, v_N), \quad \Theta := (\theta_1, \cdots, \theta_N), \\
& \D(X(t)):=\max_{1\leq i,j\leq N}\|x_i(t)-x_j(t)\|,\quad  \D(V(t)):=\max_{1\leq i,j\leq N}\|v_i(t)-v_j(t)\|, \\
& \D(\Theta(t)):=\max_{1\leq i,j\leq N}|\theta_i(t)-\theta_j(t)|.
\end{aligned}
\end{align*}
\begin{definition} \label{D1.1}
\emph{\cite{H-K-R, H-R}}
Let ${\mathcal C} := (X, V, \Theta)$ be a time-dependent configuration on the extended state space $\bbr^{Nd} \times \bbr^{Nd} \times \bbr^d$. Then the configuration ${\mathcal C}$ exhibits asymptotic mono-cluster flocking, if and only if the following conditions hold: 
\[ \sup_{0 \leq t < \infty} \D(X(t)) < \infty, \quad \lim_{t \to \infty}  \D(V(t))= 0, \quad   \lim_{t \to \infty}  \D(\Theta(t)) = 0. \]
\end{definition}

\vspace{0.2cm}

The main results of this paper are two-fold. Our first result is concerned with a sufficient framework leading to the asymptotic emergence of mono-cluster flocking. Our proof can be split into several steps. In the first step, we show that  under the assumption $ \min_{1\leq i \leq N} \theta_i^{in} > 0$, temperatures are away from zero (Lemma \ref{L3.1}). In the second step, under the a priori assumption on the boundedness of spatial diameter:
\begin{equation} \label{A-2}
 \sup_{0\leq t<\infty}\D(X(t)) < \infty,
\end{equation}
we show that the temperature and velocity diameters decay exponentially (Proposition \ref{P3.1} and Proposition \ref{P3.2}): for some $c>0$ we have
\[  \D(\Theta(t))=\mathcal O(e^{-ct})\quad\mbox{and}\quad   \D(V(t))=\mathcal O(e^{-ct})\quad\mbox{as}\quad  t\to\infty \]
In our final step, we show that if initial data $(X^{in}, V^{in}, \Theta^{in})$ satisfy
\[ \D(X^{in}) + \D(V^{in}) + \D(\Theta^{in}) \ll 1. \]
then, we can show that the a priori condition \eqref{A-2} is attained and conclude the emergence of mono-cluster flocking in the sense of Definition \ref{D1.1}. \newline

Our second main result is concerned with a sufficient framework for the mono-cluster flocking of the discrete model \eqref{A-1}. Flocking analysis for the discrete model is almost parallel to the continuous one except one extra condition on the smallness of time-step:
\[ 0 < h  \ll  \min \Big \{ \frac{1}{\kappa_1}, \frac{1}{\kappa_2} \Big \}. \]
Other precedures are almost the same as the continuous one. \newline

The rest of the paper is organized as follows: In Section \ref{sec:2}, we briefly review directed graphs, scrambling matrices, state-transition matrices and reformulation of \eqref{A-0} and \eqref{A-1} in terms of coldness which is the inverse of temperature. In Section \ref{sec:3}, we present our first result which concerns the emergence of mono-cluster flocking to the continuous model following procedure depicted as above. In Section \ref{sec:4}, we perform a similar analysis as in Section \ref{sec:3} to prove our second result which concerns the discrete model. Finally, Section \ref{sec:5} is devoted to the brief summary and discussion of our main results and some future directions.  In \ref{App-A}, \ref{App-B} and \ref{App-C}, we provide the proofs of Lemma \ref{L4.2}, Lemma \ref{L4.3} and Proposition \ref{P4.2}, respectively. 
\vspace{0.5cm}

\noindent \textbf{Notation}: For $s\in\bbr$, $\lfloor s\rfloor$ denotes the greatest integer not exceeding $s$. And for any matrix $A$, $\|A\|_F$ denotes the Frobenius norm. Throughout the paper, we denote particle $i$ by the $i$-th thermodynamic Cucker-Smale particle. For an $m\times n$ matrices $A=(a_{ij})$ and $B=(b_{ij})$, $A\geq B$ means that $a_{ij}\geq b_{ij}$ for all $i,~j$. And $A\geq0$ means that $a_{ij}\geq0$ for all $i,~j$.  An $N \times N$ matrix $A=(a_{ij})$ is nonnegative means that all entries are nonnegative. 

\section{Preliminaries} \label{sec:2}
\setcounter{equation}{0}
In this section, we provide several elementary concepts on directed graphs, scrambling matrices, state-transition matrices and a reformulation of the TCS model using coldness variable instead of temperature.

\subsection{A directed graph}\label{sec:2.1}
A directed graph(digraph) $\G=(\V(\G),\E(\G) )$ consists of a finite set $\V(\G)=\{1,\ldots, N\}$ of {\em vertices (nodes)}, and a set $\E(\G) \subset\V(\G) \times \V(\G)$ of {\em arcs}. \newline
If $(j,i)\in\E(\G)$, we say that $j$ is a {\em neighbor} of $i$, and we denote the neighbor set of the vertex $i$ by
$\N_i:=\{j:(j,i)\in\E(\G)\}$. We define the {\em adjacency matrix} $\chi=(\chi_{ij})$, where
\[ \chi_{ij} = \begin{cases}
1, \quad & \mbox{if $j$ is a neighbor of $i$}, \cr
0, \quad & \mbox{otherwise.}
\end{cases}
\]
A {\em path} in a digraph $\G$ from $i_0$ to $i_p$ is a finite
sequence $i_0, i_1, \ldots,i_p$ of distinct vertices such that each successive pair of
vertices is an arc of $\G$. The integer $p$ (the number of its
arcs) is called the {\em length} of the path.  If there exists a path from $i$ to $j$, then vertex~$j$ is said to be
{\em reachable} from vertex~$i$, and we define the distance from $i$ to $j$, $\operatorname{dist}(i,j)$,
as the length of a shortest path from $i$ to $j$.
We say that $\G$ has a {\em spanning tree} if we can find a vertex
(called a {\em root}) such that any
other vertex of~$\G$ is reachable from it.  For each root $r$ of digraph $\G$ with a spanning tree, we define $\max_{j\in\V}\operatorname{dist}(r,j)$
as the depth of the spanning tree of $\G$ rooted at $r$. The {\em smallest depth} $\gamma_g$ of $\G$ is given by the following relation:
\begin{equation*}
 \gamma_g:=\min_{r~ \mbox{is a root}}\max_{j\in\V}\operatorname{dist}(r,j).
\end{equation*}
Throughout the paper, we denote $\gamma_g$ by the smallest depth of the directed graph corresponding to the network topology $(\chi_{ij})$ in \eqref{A-0} and \eqref{A-1}.

\subsection{Scrambling matrices} Next, we review the concept of scrambling matrix and its properties.  First, we introduce concepts of stochastic matrix, scrambling matrix and adjacency matrix as follows.
\begin{definition} \label{D2.1}
Let $A=(a_{ij})$ be a nonnegative $N \times N$ matrix. 
\begin{enumerate}
\item
$A$ is a {\em stochastic} matrix, if its row-sum is equal to unity:
\[ \sum_{j=1}^{N} a_{ij} = 1, \quad 1 \leq i \leq N. \]
\item
$A$ is a {\em scrambling} matrix, if for each pair of indices $i$ and $j$, there exist an index $k$ such that $a_{ik}>0$ and $a_{jk}>0$. 
\item
$A$ is an {\em adjacency matrix} of a digraph $\G$ if the following holds: 
\[ a_{ij}>0 \quad \Longleftrightarrow \quad (j,i)\in\E. \]
In this case, we write $\G=\G(A)$.
\end{enumerate}
\end{definition}
\begin{remark}
Define the {\em ergodicity coefficient} of $A$ as follows.
\begin{equation} \label{B-1}
\mu(A):=\min_{i,j}\sum_{k=1}^N \min\{a_{ik},a_{jk}\}.
\end{equation}
Then, it is easy to see that
\begin{enumerate}
\item
$A$ is scrambling if and only if $\mu(A)>0$. 
\item
For nonnegative matrices $A$ and $B$, 
\begin{equation} \label{B-1-1}
 A\geq B \quad \Longrightarrow \quad \mu(A)\geq\mu(B). 
\end{equation} 
\end{enumerate}
\end{remark}
For a matrix $A=(a_{ij})$, we set 
\[ \underline a:=\min\{a_{ij}:a_{ij}>0\}. \]
\begin{lemma} \label{L2.1}  
\emph{(Proposition 1, \cite{DQ17})}
Let $A=(a_{ij})$ be a nonnegative $N\times N$ matrix of which all diagonal entries are positive. Suppose that $\G(A)$ has a spanning tree with the smallest depth $\gamma_g$. Then, we have 
\[ \mu(A^{\gamma_g})\geq \underline{a}^{\gamma_g}. \]
\end{lemma}
The following result is a perturbative version of Lemma 2.1 in \cite{C-D}.  
\begin{lemma}\label{L2.2}
Suppose that a nonnegative $N\times N$ matrix $A=(a_{ij})$ is stochastic, and let $B=(b_i^j)$, $Z=(z_i^j)$ and $W=(w_i^j)$ be $N\times d$ matrices such that 
\begin{equation} \label{B-2}
 W = AZ+B. 
\end{equation} 
Then, we have
\begin{equation*}
\max_{i,k}\|w_i-w_k\|\leq(1-\mu(A))\max_{l,m}\|z_l-z_m\|+\sqrt{2}\|B\|_F,
\end{equation*}
where
\[
z_i:=(z_i^1,\cdots,z_i^d),\quad b_i:=(b_i^1,\cdots,b_i^d),\quad  w_i:=(w_i^1,\cdots,w_i^d),\quad i=1,\cdots,N.
\]
\end{lemma}
\begin{proof} First we use a property of stochastic matrices  to see
\begin{equation} \label{B-2-1}
\sum_{l=1}^N\max\{0,a_{il}-a_{kl}\}+\sum_{l=1}^N\min\{0,a_{il}-a_{kl}\}=\sum_{l=1}^N(a_{il}-a_{kl})=0.
\end{equation}
Then, for $1\leq i,k\leq N$, we use \eqref{B-2}, \eqref{B-2-1} and Cauchy-Schwarz's inequality to find
\begin{align}\label{B-3}
\begin{aligned}
&\|w_i-w_k\|^2 \\
& \hspace{0.5cm} =\bigg\langle \sum_{l=1}^N a_{il}z_l+b_i-\sum_{l=1}^N a_{kl}z_l-b_k,w_i-w_k \bigg\rangle\\
& \hspace{0.5cm} =\sum_{l=1}^N (a_{il}-a_{kl})\langle z_l,w_i-w_k\rangle+\langle b_i-b_k,w_i-w_k\rangle\\
& \hspace{0.5cm} \leq\sum_{l=1}^N\max\{0,a_{il}-a_{kl}\}\max_n \langle z_n,w_i-w_k\rangle \\
& \hspace{0.8cm}+\sum_{l=1}^N\min\{0,a_{il}-a_{kl}\}\min_n \langle z_n,w_i-w_k\rangle +\langle b_i-b_k,w_i-w_k\rangle\\
& \hspace{0.5cm} =\sum_{l=1}^N\max\{0,a_{il}-a_{kl}\}\max_{n,m} \langle z_n-z_m,w_i-w_k\rangle+\langle b_i-b_k,w_i-w_k\rangle\\
& \hspace{0.5cm} \leq\sum_{l=1}^N\max\{0,a_{il}-a_{kl}\}\max_{n,m} \| z_n-z_m\|\|w_i-w_k\|+\| b_i-b_k\|\|w_i-w_k\|,
\end{aligned}
\end{align}
By \eqref{B-3}, we have
\[
\max_{i,k}\|w_i-w_k\|\leq \max_{i,k}\sum_{l=1}^N\max\{0,a_{il}-a_{kl}\}\max_{n,m} \| z_n-z_m\|+\max_{i,k}\|b_i-b_k\|.
\]
Finally, we use the following observations:
\[
\max_{i,k}\sum_{l=1}^N\max\{0,a_{il}-a_{kl}\}=\max_{i,k}\sum_{l=1}^N (a_{il}-\min\{a_{il},a_{kl}\})=1-\mu(A)
\]
and
\[
\|b_i-b_k\|^2\leq\big(\|b_i\|+\|b_k\|\big)^2\leq2\|b_i\|^2+2\|b_k\|^2\leq2\|B\|_F^2\quad \mbox{for}\quad i\neq k
\]
to derive the desired estimate.
\end{proof}
\subsection{State-transition matrices}
Let $t_0\in\bbr$ and $A:[t_0,\infty)\to\bbr^{N\times N}$ be an $N\times N$ matrix of continuous functions. Consider the following time-dependent linear ODE:
\begin{equation}\label{B-4}
\frac{d\xi(t)}{dt} =A(t)\xi(t),\quad t > t_0.
\end{equation}
Then, the solution of \eqref{B-4} is given by 
\begin{equation*} 
\xi(t)=\Phi(t,t_0)\xi(t_0), 
\end{equation*}
where $\Phi(t,t_0)$ is the state-transition matrix or fundamental matrix. We can write the state-transition matrix $\Phi(t,t_0)$ in the following form, which is known as the Peano-Baker series (see \cite{Sontag}):
\[
\Phi(t,t_0)=I+\sum_{n=1}^\infty \int_{t_0}^t\int_{t_0}^{\tau_1}\cdots\int_{t_0}^{\tau_{n-1}}A(\tau_1)A(\tau_{2})\cdots A(\tau_n)d\tau_n\cdots d\tau_2d\tau_1,
\]
where $I$ is $N \times N$ identity matrix. We conclude this subsection by introducing a technical lemma to be used later. \newline

Let $t_0\in\bbr$, $c\in\bbr$ and $A:[t_0,\infty)\to\bbr^{N\times N}$ be an $N\times N$ matrix of continuous functions. Then, we set $\Phi(t,t_0)$ and $\Psi(t,t_0)$ to be the state-transition matrices corresponding to the following linear ODEs, respectively:
\begin{equation} \label{B-4-0}
\frac{d\xi(t)}{dt} =A(t)\xi(t) \quad \mbox{and} \quad 
\frac{d\xi(t)}{dt} =[A(t)+cI]\xi(t),\quad t\geq t_0.
\end{equation}
Then, the following lemma yields a relation between $\Phi(t,t_0)$ and $\Psi(t,t_0)$.
\begin{lemma}\label{L2.3}
The following relation holds.
\[ \Phi(t,t_0)=e^{-c(t-t_0)}\Psi(t,t_0), \quad \mbox{or} \quad \Psi(t,t_0) =e^{c(t-t_0)}  \Phi(t,t_0), \quad t \geq t_0. \]
\end{lemma}
\begin{proof} Let $\Phi(t,t_0)$ and $\Psi(t,t_0)$ be the state transition matrices of $\eqref{B-4-0}_1$ and $\eqref{B-4-0}_2$, respectively. To derive desired estimate, we set
\begin{equation} \label{B-4-1}
 \tilde\Phi(t,t_0):=e^{-c(t-t_0)}\Psi(t,t_0). 
\end{equation} 
and we will show that $ \tilde\Phi(t,t_0)$ satisfies  $\eqref{B-4-0}_1$ and the same initial data.  \newline

\noindent $\bullet$ (Equation): By direct estimate, we have
\begin{align}
\begin{aligned} \label{B-5}
\frac{d}{dt}\tilde\Phi(t,t_0)&=\frac{d}{dt}\big[e^{-c(t-t_0)}\Psi(t,t_0)\big] \\
&=-ce^{-c(t-t_0)}\Psi(t,t_0)+e^{-c(t-t_0)}\frac{d}{dt}\Psi(t,t_0)\\
&=-ce^{-c(t-t_0)}\Psi(t,t_0)+e^{-c(t-t_0)}[A(t)+cI]\Psi(t,t_0)\\
&=e^{-c(t-t_0)}A(t)\Psi(t,t_0)=A(t)\tilde\Phi(t,t_0).
\end{aligned}
\end{align}

\noindent $\bullet$ (Initial data): It is clear from \eqref{B-4-1} that 
\begin{equation}\label{B-6}
\tilde\Phi(t_0,t_0)=\Psi(t_0,t_0)=I.
\end{equation}
Finally, we can see that \eqref{B-5} and \eqref{B-6} satisfies the same ODE system and initial data. By the uniqueness theory of ODEs, we have
\[ \Phi(t,t_0) =  \tilde\Phi(t,t_0). \]
\end{proof}

\subsection{A reformulation of the TCS model} In this subsection, we introduce a ``{\it coldness}" variable which is a reciprocal of the temperature. We set
\[
\beta_i(t):=\frac{1}{\theta_i(t)},~~t > 0, \qquad  \beta_i^{in}:=\frac{1}{\theta_i^{in}},~~ i=1,\cdots,N.
\]
Then $\beta$ measures the inverse tempertaure, i.e. coldness. Then, the corresponding Cauchy problem for the continuous and discrete models \eqref{A-0} and \eqref{A-1} are
\begin{equation} \label{B-10}
\begin{cases}
\displaystyle \frac{dx_i}{dt}=v_i,\quad t>0, \quad i=1,2,\cdots, N,\\
\displaystyle \frac{dv_i}{dt}=\frac{1}{N}\sum_{j=1}^N \chi_{ij}\phi(\|x_i-x_j\|) (\beta_j v_j -\beta_iv_i),\\
\displaystyle \frac{d\beta_i}{dt}=\frac{1}{N}\sum_{j=1}^N \chi_{ij}\zeta(\|x_i -x_j\|)\beta_i^2 (\beta_j -\beta_i),\\
(x_i(0),v_i(0),\beta_i(0))=(x_i^{in},v_i^{in},\beta_i^{in}),
\end{cases}
\end{equation}
and
\begin{equation}\label{B-11}
\begin{cases}
\displaystyle x_i[t+1]=x_i[t]+h v_i[t],\quad t\in \bbn\cup\{0\}, \quad i=1,2,\cdots, N,\\
\displaystyle v_i[t+1]=v_i[t]+\frac{h}{N}\sum_{j=1}^N \chi_{ij}\phi(\|x_i[t]-x_j[t]\|)\Big(\beta_j[t]v_j[t]-\beta_i[t]v_i[t]\Big),\\
\displaystyle \frac{1}{\beta_i[t+1]} =\frac{1}{\beta_i[t]} +\frac{h}{N}\sum_{j=1}^N \chi_{ij}\zeta(\|x_i[t]-x_j[t]\|) \Big(\beta_i[t]-\beta_j[t]\Big),\\
\displaystyle (x_i[0],v_i[0],\beta_i[0])=(x_i^{in},v_i^{in},\beta_i^{in}).
\end{cases}
\end{equation}

\subsection{A review of previous results} In this subsection, we briefly review the known previous results on the emergent behaviors of the TCS model \eqref{A-0} with small diffusion velocities. The model \eqref{A-0} with all-to-all couplings with $\phi \equiv 1$ and $\zeta \equiv 1$ has been proposed in a recent work by Ha and Ruggeri \cite{H-R}. They derived the model \eqref{A-0} from the system of gas mixture which is a coupled system of reactive Euler systems based on reasonable physical settings such as spatial homogeneity, Galilean invariance, small diffusion velocity assumption and entropy principle (see \cite{H-R} for a detailed discussion). In their work, they derived an exponential flocking estimate for \eqref{A-0} as long as initial states satisfy a kind of small assumptions. Their work has been extended to several directions, e.g., nonexistence of mono-cluster flocking  \cite{H-K-R}, time-delay effect \cite{D-H-K-K}, uniform stability and its kinetic limit \cite{H-K-M-R-Z2}, global well-posedness of the hydrodynamic TCS model in \cite{H-K-M-R-Z1}, coupling with fluids 
\cite{C-H-J-K1, C-H-J-K2}, disctete TCS model \cite{H-K-L}, although they are all dealing with TCS ensemble over the complete graph. As far as the authors know, this is the first work dealing with the emergent dynamics of TCS ensemble other than the complete graph. 

\section{Emergence dynamics of the continuous model} \label{sec:3}
\setcounter{equation}{0}
In this section, we present an asymptotic flocking estimate for the continuous model \eqref{B-10}. In the sequel, we derive our flocking estimate in the following three steps.
\begin{itemize}
\item
First, we show the uniform boundedness of temperatures and monotonic properties of the diameter of the coldness, and  using these estimates, we derive an exponential asymptotic alignment of temperatures under a priori uniform boundedness condition of spatial diameter.

\vspace{0.2cm}

\item
Second, we derive asymptotic alignment of velocities under a priori uniform boundedness condition of spatial diameter.

\vspace{0.2cm}

\item
Finally, we present a sufficient condition leading to the uniform boundedness of spatial diameter in terms of system parameters. This leads to the mono-cluster flocking estimate.
\end{itemize}

\subsection{A priori temperature alignment} In this subsection, we show that the temperatures have some positive lower bound and upper bound. For convenience, we work with the system \eqref{B-10}. Next, we set initial maximum and minimum coldness $\beta^{in}_U$ and $\beta^{in}_L>0$ as follows:
\[
\beta_U^{in}:=\max_{1\leq i\leq N}\beta_i^{in} \qquad \mbox{and} \qquad \beta_L^{in}:=\min_{1\leq i\leq N}\beta_i^{in}.
\]
We also set
\[ {\mathcal B} := (\beta_1, \cdots, \beta_N),\quad \mathcal D(\mathcal B(t)):=\max_{1\leq i,j\leq N}|\beta_i(t)-\beta_j(t)|. \]
\begin{lemma}[Boundedness of temperatures]\label{L3.1}
Let $\{ (x_i,v_i,\beta_i) \}$ be a solution to the Cauchy problem \eqref{B-10} with positive initial temperatures $\beta^{in}_U < \infty$.  Then, we have 
\begin{eqnarray*}
&& (i)~\beta_L^{in} \leq \beta_i(t) \leq \beta_U^{in},\quad i=1,\cdots,N,\quad 0\leq t<\infty,\\
&& (ii)~\D({\mathcal B}(\cdot)) \mbox{ is monotone decreasing}.
\end{eqnarray*}
\end{lemma}
\begin{proof}
(i)~First, we define the maximal and minimal values for coldness as follows: 
\[
\beta_M(t):=\max_{1\leq i\leq N}\beta_i(t),\quad \beta_m(t):=\min_{1\leq i\leq N}\beta_i(t),\quad t\geq 0.
\]
Then, for each $t>0$, we choose extremal indices $1\leq m_t, ~M_t\leq N$ satisfying
\[  \beta_m(t)=\beta_{m_t}(t) \quad \mbox{and} \quad \beta_M(t)=\beta_{M_t}(t). \]
Note that for a.e. $t > 0$, $\beta_m$ and $\beta_M$ satisfy 
\begin{align}
\begin{aligned} \label{C-1}
\frac{d\beta_m(t)}{dt} &=\frac{\beta_{m_t}(t)^2}{N}\sum_{j=1}^N \chi_{{m_t}j}\zeta(\|x_{m_t}(t)-x_j(t)\|)\Big(\beta_j(t)-\beta_{m_t}(t)\Big)\geq0, \\
\frac{d\beta_M(t)}{dt} &=\frac{\beta_{M_t}(t)^2}{N}\sum_{j=1}^N \chi_{{M_t}j}\zeta(\|x_{M_t}(t)-x_j(t)\|)\Big(\beta_j(t)-\beta_{M_t}(t)\Big)\leq0.
\end{aligned}
\end{align}
Thus, minimal and maximal coldness are non-decreasing and non-increasing along the flow \eqref{B-10}.

\vspace{0.2cm}

\noindent (ii)~Recall the diameter for coldness:
\[ \D({\mathcal B}(t)) = \max_{1\leq i \leq N} \beta_i(t) - \min_{1\leq i \leq N} \beta_i(t), \quad  t \geq 0. \]
We use the above defining relation and \eqref{C-1} to get the desired estimate:
\[
\frac{d}{dt}\D({\mathcal B}(t))=\frac{d}{dt}\big(\beta_M(t)-\beta_m(t)\big)\leq0,\quad \mbox{a.e. }t>0.
\]
\end{proof}
Next, we study the exponential decay of $\D({\mathcal B})$ using a more refined argument.  First, we rewrite $\eqref{B-10}_3$ as follows.
\begin{equation}  \label{C-1-1}
\frac{d\beta_i}{dt}=-\frac{1}{N} \beta_i^2 \Big[ \beta_i  \Big( \sum_{j=1}^N \chi_{ij}\zeta(\|x_i -x_j\|) \Big) - \sum_{j=1}^N \chi_{ij}\zeta(\|x_i -x_j\|) \beta_j \Big]
\end{equation}
In order to rewrite \eqref{C-1-1} in a more compact form, we  define an $N\times N$ matrix $L(t)$ by
\begin{equation} \label{C-1-2}
L(t):=D(t)-A(t),
\end{equation}
where the matrices $A(t)=(a_{ij}(t))$ and $D(t)=\operatorname{diag}(d_1(t),\cdots,d_N(t))$ are defined by the following relations:
\begin{equation} \label{C-1-3}
a_{ij}(t):=\chi_{ij}\zeta(\|x_i(t)-x_j(t)\|)\quad\mbox{and}\quad d_i(t)=\sum_{j=1}^N \chi_{ij}\zeta(\|x_i(t)-x_j(t)\|).
\end{equation}
And we also define
\begin{equation} \label{C-1-4}
\Gamma(t):=\mbox{diag}\big(\beta_1(t),\cdots,\beta_N(t)\big).
\end{equation}
Then, we see from \eqref{C-1-1} that $\eqref{B-10}_3$ can be written as 
\begin{equation} \label{C-2}
\frac{d}{dt} {\mathcal B}(t) =-\frac{1}{N}\Gamma(t)^2L(t) {\mathcal B}(t).
\end{equation}
Let $\Phi(t_2,t_1)$ be the state transition matrix associated with \eqref{C-2}. Then, for any given $\delta > 0$, we derive the solution formula for ${\mathcal B}$:
\begin{equation} \label{C-3}
 {\mathcal B}(m\delta) = \Phi\Big(m\delta,(m-1)\delta\Big) {\mathcal B}((m-1) \delta), \quad  m \in \bbn. 
\end{equation} 

\begin{lemma} \label{L3.2}
Let $\{ (x_i,v_i,\beta_i)\}$ be a solution to \eqref{B-10} satisfying a priori condition:
\begin{equation}\label{C-10}
\sup_{0\leq t<\infty}\D(X(t))\leq x^{\infty} < \infty.
\end{equation}
Then the following assertions hold.
\begin{enumerate}
\item
The ergodicity coefficient $\mu \Big(\Phi\big(m\delta,(m-1)\delta\big)\Big)$ satisfies
\[ \mu \Big(\Phi\big(m\delta,(m-1)\delta\big)\Big) \geq   C_1 \zeta(x^{\infty})^{\gamma_g},  \]
where a positive constant $C_1$ is given by the following relation:
\[
C_1=C_1(\delta):=e^{-\kappa_2(\beta_U^{in})^2\delta}\cdot\frac{1}{\gamma_g!}\Big(\delta\frac{(\beta_L^{in})^2}{N}\Big)^{\gamma_g}.
\]
\item
The state transition matrix $ \Phi\big(m\delta,(m-1)\delta\big)$ is stochastic.
\end{enumerate}
\end{lemma}
\begin{proof} 
\noindent (1)~We claim
\begin{equation} \label{C-4}
 \Phi\Big(m\delta,(m-1)\delta\Big) \geq C_1 (A^{\infty})^{\gamma_g}\geq0, 
\end{equation} 
 where $A^{\infty}=(a_{ij}^{\infty})$ is a nonnegative matrix defined by $a_{ij}^{\infty}:=\chi_{ij}\zeta(x^{\infty}).\\$
{\it Proof of claim \eqref{C-4}}:  We use \eqref{C-1-2}, \eqref{C-1-3} and \eqref{C-1-4} to estimate
 the coefficient matrix for \eqref{C-2} as follows.
\begin{equation*}
-\frac{1}{N}\Gamma(t)^2L(t)=\frac{1}{N}\Gamma(t)^2(A(t)-D(t))\geq\frac{(\beta_L^{in})^2}{N}A^{\infty}-\kappa_2(\beta_U^{in})^2I.
\end{equation*}
Now, we decompose the coefficient matrix $-\frac{1}{N}\Gamma(t)^2L(t)$ into the sum of the following two matrices:
\begin{equation} \label{C-5-1}
-\frac{1}{N}\Gamma(t)^2L(t)=\left(-\frac{1}{N}\Gamma(t)^2L(t)+\kappa_2(\beta_U^{in})^2I\right)-\kappa_2(\beta_U^{in})^2I.
\end{equation}
The terms in the parenthesis of \eqref{C-5-1} can also be estimated as follows:
\begin{equation}\label{C-6}
-\frac{1}{N}\Gamma(t)^2L(t)+\kappa_2(\beta_U^{in})^2I\geq \frac{(\beta_L^{in})^2}{N}A^{\infty} \geq0.
\end{equation}
On the other hand, for any $0\leq t_1<t_2<\infty$, let $\Phi(t_2,t_1)$ and $\Psi(t_2,t_1)$ be the state-transition matrices of $-\frac{1}{N}\Gamma(t)^2L(t)$ and $-\frac{1}{N}\Gamma(t)^2L(t)+\kappa_2(\beta_U^{in})^2I$ on $[t_1,t_2]$, respectively. 
Then it follows from Lemma \ref{L2.3} that
\begin{equation} \label{C-7}
\Phi\left(t_2,t_1\right)=e^{-\kappa_2(\beta_U^{in})^2(t_2-t_1)}\Psi\left(t_2,t_1\right).
\end{equation}
In \eqref{C-6}, we can use the Peano-Baker series to obtain
\begin{align} \label{C-8}
\begin{aligned}
\Psi\left(t_2,t_1\right)&=I+\sum_{n=1}^\infty \int_{t_1}^{t_2}\int_{t_1}^{\tau_1}\cdots\int_{t_1}^{\tau_{n-1}}\Big(\big(-\frac{1}{N}\Gamma(\tau_1)^2L(\tau_1)+\kappa_2(\beta_U^{in})^2I\big)\cdots \\
&\hspace{2cm}\big(-\frac{1}{N}\Gamma(\tau_n)^2L(\tau_n)+\kappa_2(\beta_U^{in})^2I\big) \Big)d\tau_n\cdots d\tau_1\\
&\geq I+\sum_{n=1}^\infty \int_{t_1}^{t_2}\int_{t_1}^{\tau_1}\cdots\int_{t_1}^{\tau_{n-1}}\Big(\frac{(\beta_L^{in})^2}{N}A^{\infty} \Big)^n d\tau_n\cdots d\tau_1\\
&= I+\sum_{n=1}^\infty\frac{1}{n!} (t_2-t_1)^n\Big(\frac{(\beta_L^{in})^2}{N}A^{\infty} \Big)^n \\
&= \exp\left((t_2-t_1)\frac{(\beta_L^{in})^2}{N}A^{\infty} \right).
\end{aligned}
\end{align}
For a fixed $m\in\bbn$, we combine \eqref{C-7} and \eqref{C-8} and put $t_1=(m-1)\delta$, $t_2=m\delta$ to obtain 
\begin{align}
\begin{aligned} \label{C-8-1}
\Phi\Big(m\delta,(m-1)\delta\Big)&\geq e^{-\kappa_2(\beta_U^{in})^2\delta} \exp\Big[\delta\frac{(\beta_L^{in})^2}{N}A^{\infty} \Big] \\
&=e^{-\kappa_2(\beta_U^{in})^2\delta}\Big[I+\sum_{n=1}^\infty\frac{1}{n!}\Big(\delta\frac{(\beta_L^{in})^2}{N}A^{\infty} \Big)^n\Big]\\
&\geq e^{-\kappa_2(\beta_U^{in})^2\delta}\cdot\frac{1}{\gamma_g!}\Big(\delta\frac{(\beta_L^{in})^2}{N}\Big)^{\gamma_g} (A^{\infty})^{\gamma_g} \\
&= C_1 (A^{\infty})^{\gamma_g}\geq0,
\end{aligned}
\end{align}
which proves the claim \eqref{C-4}.\\  \eqref{C-8-1} and \eqref{B-1-1} yield
\begin{equation*}
\mu\Big(\Phi\Big(m\delta,(m-1)\delta\Big)\Big)\geq C_1\mu((A^{\infty})^{\gamma_g})\geq C_1 \zeta(x^{\infty})^{\gamma_g}.
\end{equation*}
where $\mu$ is the ergodicity coefficient defined in \eqref{B-1}, and the last inequality is due to Lemma \ref{L2.1}.\newline

(2)~By \eqref{C-8-1} $\Phi\big(m\delta,(m-1)\delta\big)$ is nonnegative, so it remains to show that each of its rows sums to 1. Note that the constant state $\xi(t):=[\xi_1(t),\cdots,\xi_N(t)]^\top\equiv [1,\cdots,1]^\top$ is a solution to $\eqref{C-2}$, i.e.
\[
\frac{d}{dt} {\xi}(t) =-\frac{1}{N}\Gamma(t)^2L(t) {\xi}(t).
\]
Hence, it satisfies \eqref{C-3}:
\[ [1,\cdots,1]^\top=\Phi(m\delta,(m-1)\delta)[1,\cdots,1]^\top. \]
This implies that $\Phi\big(m\delta,(m-1)\delta\big)$ is stochastic. 
 
\end{proof}
Next, we are ready to present a priori temperature alignment based on Lemma \ref{L3.2}. 
\begin{proposition}\label{P3.1}
Let $\{ (x_i,v_i,\beta_i)\}$ be a solution to \eqref{B-10} satisfying a priori condition \eqref{C-10}.
Then, we have the exponential decay of $\D({\mathcal B}(t))$: For any given $\delta>0$, we have
\[
\D({\mathcal B}(t))\leq \Big(1-C_1\zeta(x^{\infty})^{\gamma_g} \Big)^{\lfloor \frac{t}{\delta}\rfloor}\D({\mathcal B}(0)),\quad t\geq0,
\]
where $C_1=C_1(\delta)$ is the constant defined in Lemma \ref{L3.2}.
\end{proposition}
\begin{proof} 
Since $\Phi\big(m\delta,(m-1)\delta\big)$ is stochastic (Lemma \ref{L3.2}), we can combine \eqref{C-3}, Lemma \ref{L2.2} and Lemma \ref{L3.2} to obtain
\begin{align}
\begin{aligned} \label{C-11}
\D\left({\mathcal B}(m\delta)\right)&\leq\left(1-\mu\Big(\Phi\Big(m\delta,(m-1)\delta\Big)\Big)\right)\D({\mathcal B}((m-1)\delta))\\
&\leq \left(1-C_1\zeta(x^\infty)^{\gamma_g}\right)\D({\mathcal B}((m-1)\delta))\\
&\leq\cdots\leq \left(1-C_1\zeta(x^\infty)^{\gamma_g}\right)^m\D({\mathcal B}(0)),\quad m\in\bbn.
\end{aligned}
\end{align}
So for any real $t=\delta p\geq0$, we use \eqref{C-11} and Lemma \ref{L3.1} to get
\begin{align*}
\begin{aligned}
\D({\mathcal B}(t))&=\D\left({\mathcal B}\Big(\delta p\Big)\right)\leq \D\left({\mathcal B}\Big(\delta\lfloor p\rfloor\Big)\right) 
\leq \left(1-C_1\zeta(x^\infty)^{\gamma_g}\right)^{\lfloor p\rfloor}\D({\mathcal B}(0))\\
&= \left(1-C_1\zeta(x^\infty)^{\gamma_g}\right)^{\lfloor \frac{t}{\delta}\rfloor}\D({\mathcal B}(0)).
\end{aligned}
\end{align*}
\end{proof}
\subsection{A priori velocity alignment} In this subsection, we provide velocity alignment estimate under the a priori assumption \eqref{C-10}. For notational simplicity, we introduce the following notation:
\[
u_i(t):=\frac{v_i(t)}{\theta_i(t)}=\beta_i(t)v_i(t),\quad i=1,\cdots,N,\quad \mbox{and}\quad R_u(t):=\max_{1\leq i\leq N}\|u_i(t)\|,\quad t\geq0.
\]
To derive the velocity alignment, we use a bootstrapping argument, i.e., first we derive a uniform boundedness of velocity diameter, and then using the differential inequalities for velocity diameter, we improve our rough boundedness to the exponential decay of the velocity diameter.  As a first step, we prove the boundedness of velocities.
\begin{lemma}[Boundedness of velocities] \label{L3.3}
Let $\{ (x_i,v_i,\beta_i)\}$ be a solution to \eqref{B-10} satisfying a priori condition \eqref{C-10}. Then, velocities of the particles are uniformly bounded: for any giiven $\delta>0$ we have
\[
\|v_i(t)\|\leq \frac{1}{\beta_L^{in}}R_u(0)\exp\left(\frac{\kappa_2\delta\beta^{in}_U\D({\mathcal B}(0))}{C_1\zeta(x^{\infty})^{\gamma_g}}\right)=:R_V^c=R_V^c(x^{\infty},\delta),\quad i=1,\cdots,N,\quad t\geq0,
\]
where $C_1=C_1(\delta)$ is the constant defined in Lemma \ref{L3.2}.
\end{lemma}
\begin{proof} For the desired estimate,  it suffices to derive an estimate:
\[
R_u(t)\leq R_u(0)\exp\left(\frac{\kappa_2\delta\beta_U^{in}\D({\mathcal B}(0))}{C_1\zeta(x^{\infty})^{\gamma_g}}\right),\quad t\geq0.
\]
First, we derive a differential inequality of $R_u$. For each $i=1,\cdots,N$, we have
\begin{equation}\label{C-11-1}
\frac{d}{dt}\|u_i\|^2=\frac{d}{dt}(\beta_i^2\|v_i\|^2)=2\beta_i\|v_i\|^2\frac{d\beta_i}{dt} + 2\beta_i^2\Big\langle v_i,\frac{d v_i}{dt}\Big\rangle,\quad t>0.
\end{equation}
We estimate the two terms of the right-hand side of \eqref{C-11-1}. We have
\begin{align}\label{C-12}
\begin{aligned}
2\beta_i\|v_i\|^2\frac{d\beta_i}{dt}&=\frac{2\beta_i\|v_i\|^2}{N}\sum_{j=1}^N \chi_{ij}\zeta(\|x_i-x_j\|)\beta_i^2(\beta_j-\beta_i)\\
&=\frac{2\beta_i\|u_i\|^2}{N}\sum_{j=1}^N \chi_{ij}\zeta(\|x_i-x_j\|)(\beta_j-\beta_i)\\
&\leq\frac{2\beta_i\|u_i\|^2}{N}\sum_{j=1}^N \kappa_2\D({\mathcal B}) =2\kappa_2\beta_i\|u_i\|^2\D({\mathcal B}),
\end{aligned}
\end{align}
and
\begin{align}\label{C-13}
\begin{aligned}
2\beta_i^2\Big\langle v_i,\frac{d v_i}{dt}\Big\rangle&=2\beta_i^2\left\langle v_i,\frac{1}{N}\sum_{j=1}^N \chi_{ij}\phi(\|x_i-x_j\|)\Big(\beta_jv_j-\beta_iv_i\Big)\right\rangle\\
&=2\beta_i\left\langle u_i,\frac{1}{N}\sum_{j=1}^N \chi_{ij}\phi(\|x_i-x_j\|)\Big(u_j-u_i\Big)\right\rangle\\
&=\frac{2\beta_i}{N}\sum_{j=1}^N \chi_{ij}\phi(\|x_i-x_j\|)\Big(\langle u_i,u_j\rangle-\langle u_i,u_i\rangle\Big)\\
&\leq\frac{2\beta_i}{N}\sum_{j=1}^N \chi_{ij}\phi(\|x_i-x_j\|)\Big(\| u_i\|\|u_j\|-\|u_i\|^2\Big)\\
&\leq\frac{2\beta_i\|u_i\|}{N}\sum_{j=1}^N \chi_{ij}\phi(\|x_i-x_j\|)\big(\|u_j\|-\|u_i\|\big).
\end{aligned}
\end{align}
We combine \eqref{C-11-1}, \eqref{C-12}, and \eqref{C-13} to obtain 
\[
\frac{d}{dt}\|u_i\|^2\leq2\kappa_2\beta_i\D({\mathcal B})\|u_i\|^2+\frac{2\beta_i\|u_i\|}{N}\sum_{j=1}^N \chi_{ij}\phi(\|x_i-x_j\|)\big(\|u_j\|-\|u_i\|\big),\quad t>0.
\]
For each $t>0$, we take a maximal index $i_t$ satisfing $R_u(t)=\|u_{i_t}(t)\|$. Then we have
\begin{align*}
\begin{aligned}
\frac{d}{dt}R_u(t)^2&\leq 2\kappa_2\beta_{i_t}\D({\mathcal B})R_u(t)^2+\frac{2\beta_{i_t}R_u(t)}{N}\sum_{j=1}^N \chi_{{i_t}j}\phi(\|x_{i_t}-x_j\|)\big(\|u_j\|-R_u(t)\big)\\
&\leq2\kappa_2\beta_{i_t}\D({\mathcal B})R_u(t)^2\leq2\kappa_2\beta_U^{in}\D({\mathcal B}(t))R_u(t)^2,\quad\mbox{a.e. }t>0.
\end{aligned}
\end{align*}
Hence
\[
2R_u(t)\frac{d}{dt}R_u(t)\leq2\kappa_2\beta^{in}_U\D({\mathcal B}(t))R_u(t)^2,\quad\mbox{a.e. }t>0.
\]
If $R_u(t)>0$, then we can divide the above inequality by $2R_u(t)$. If $R_u(t)=0$, then $R_u$ attains a global minimum at $t$, so $\frac{d}{dt}R_u(t)=0$.
Hence we have the following differential inequality:
\[
\frac{d}{dt}R_u(t)\leq\kappa_2\beta^{in}_U\D({\mathcal B}(t))R_u(t),\quad\mbox{a.e. }t>0.
\]
Now we apply Gronwall's inequality and Proposition \ref{P3.1} to obtain
\begin{align*}
\begin{aligned}
R_u(t)&\leq R_u(0)\exp\left(\kappa_2\beta_U^{in}\int_0^t\D({\mathcal B}(s))ds\right)\\
&\leq R_u(0)\exp\left(\kappa_2\beta_U^{in}\D({\mathcal B}(0))\int_0^\infty\left(1-C_1\zeta(x^\infty)^{\gamma_g}\right)^{\lfloor \frac{s}{\delta}\rfloor}ds\right)\\
&= R_u(0)\exp\left(\kappa_2\beta_U^{in}\D({\mathcal B}(0))\sum_{n=0}^\infty \int_{n\delta}^{(n+1)\delta} \left(1-C_1\zeta(x^{\infty})^{\gamma_g}\right)^{\lfloor \frac{s}{\delta}\rfloor}ds\right)\\
&= R_u(0)\exp\left(\kappa_2 \beta_U^{in}\D({\mathcal B}(0))\sum_{n=0}^\infty \delta \left(1-C_1\zeta(x^{\infty})^{\gamma_g}\right)^n\right)\\
&= R_u(0)\exp\left(\frac{\kappa_2 \delta\beta_U^{in}\D({\mathcal B}(0))}{C_1\zeta(x^{\infty})^{\gamma_g}}\right).
\end{aligned}
\end{align*}
\end{proof}
Next, we derive a differential inequality for the velocity diameter.
\begin{lemma}[Differential inequality of $\D(V(t))$] \label{L3.4}
Let $\{ (x_i,v_i,\beta_i)\}$ be a solution to \eqref{B-10} satisfying a priori condition \eqref{C-10}. Then, for any given $\delta>0$ we have
\[
\frac{d}{dt}\D(V(t))\leq2\kappa_1R_V^c\D({\mathcal B}(t)),\quad \mbox{a.e. }t>0,
\]
where $R_V^c=R_V^c(x^{\infty},\delta)$ is the constant defined in Lemma \ref{L3.3}.
\end{lemma}
\begin{proof}
For a given $t$, let $i$ and $j$ be indices satisfying the relation:
\[ \D(V)=\|v_i-v_j\|. \]
Then, we have
\begin{align}
\begin{aligned}\label{C-13-1}
\frac{1}{2}\frac{d}{dt}\|v_i-v_j\|^2 &=\left\langle v_i-v_j,\frac{dv_i}{dt}-\frac{dv_j}{dt}\right\rangle\\
&=\left\langle v_i-v_j,\frac{1}{N}\sum_{k=1}^N \chi_{ik}\phi_{ik}(\beta_kv_k-\beta_iv_i)\right\rangle \\
&+\left\langle v_j-v_i,\frac{1}{N}\sum_{k=1}^N \chi_{jk}\phi_{jk}(\beta_kv_k-\beta_jv_j)\right\rangle\\
& =:\mathcal{I}_{11}+\mathcal{I}_{12},
\end{aligned}
\end{align}
where we wrote $\phi_{ij}:=\phi(\|x_i-x_j\|)$, $i,j=1,2,\cdots, N$ for notational convenience. \newline

Below, we estimate the terms ${\mathcal I}_{1i},~i=1,2$ one by one. \newline

\noindent $\bullet$ (Estimate of $\mathcal{I}_{11}$): We use $ \phi_{ik} \leq \phi(0) = \kappa_1$ to find
\begin{align*}
\begin{aligned}
\mathcal{I}_{11} &= \frac{1}{N}\sum_{k=1}^N\chi_{ik} \phi_{ik}\left\langle v_i-v_j,\beta_kv_k-\beta_iv_k\right\rangle + \frac{1}{N}\sum_{k=1}^N\chi_{ik} \phi_{ik}\left\langle v_i-v_j,\beta_iv_k-\beta_iv_i\right\rangle \\
&\leq \frac{1}{N}\sum_{k=1}^N\chi_{ik} \phi_{ik}\left\langle v_i-v_j, (\beta_k - \beta_i) v_k \right\rangle + 0 \\
& \leq \kappa_1\D({\mathcal B})\| v_i-v_j\| \Big( \frac{1}{N}\sum_{k=1}^N \|v_k\| \Big).
\end{aligned}
\end{align*}
The first inequality followed from
\begin{align*}
\begin{aligned}
\langle v_k-v_i,v_i-v_j\rangle&=\frac{\|v_k-v_j\|^2-\|v_k-v_i\|^2-\|v_i-v_j\|^2}{2}\\
&\leq \frac{\|v_i-v_j\|^2-0-\|v_i-v_j\|^2}{2}=0.
\end{aligned}
\end{align*}
\vspace{0.2cm}

\noindent $\bullet$ (Estimate of $\mathcal{I}_{12}$) : 
\begin{align*}
\begin{aligned}
\mathcal{I}_{12} &= \frac{1}{N}\sum_{k=1}^N\chi_{jk} \phi_{jk}\left\langle v_j-v_i,\beta_kv_k-\beta_jv_k\right\rangle + \frac{1}{N}\sum_{k=1}^N\chi_{jk} \phi_{jk}\left\langle v_j-v_i,\beta_jv_k-\beta_jv_j\right\rangle \\
&\leq \frac{1}{N}\sum_{k=1}^N\chi_{jk} \phi_{jk}\left\langle v_j-v_i, (\beta_k - \beta_j) v_k \right\rangle + 0 \\
&\leq \kappa_1\D({\mathcal B})\| v_i-v_j\| \Big( \frac{1}{N}\sum_{k=1}^N \|v_k\| \Big).
\end{aligned}
\end{align*}
The first inequality followed from
\begin{align*}
\begin{aligned}
\langle v_k-v_j,v_j-v_i\rangle&=\frac{\|v_k-v_i\|^2-\|v_k-v_j\|^2-\|v_j-v_i\|^2}{2}\\
&\leq \frac{\|v_j-v_i\|^2-0-\|v_j-v_i\|^2}{2}=0.
\end{aligned}
\end{align*}

Now, we combine estimates for $\mathcal{I}_{11}$ and $\mathcal{I}_{12}$ in \eqref{C-13-1} and Lemma \ref{L3.3} to obtain
\[
\mathcal{I}_{11}+\mathcal{I}_{12}\leq 2\kappa_1R_V^c\D({\mathcal B})\|v_i-v_j\|.
\]
Since $\D(V)=\|v_i-v_j\|$, we have
\[
\D(V(t))\frac{d}{dt}\D(V(t))\leq2\kappa_1R_V^c\D({\mathcal B})\D(V(t)),\quad\mbox{a.e. }t>0.
\]
If $\D(V(t))>0$, then we can divide the above inequality by $\D(V(t))$. If $\D(V(t))=0$, then $\D(V)$ attains a global minimum at $t$, so $\frac{d}{dt}\D(V(t))=0$.
Hence we have the following differential inequality:
\[
\frac{d}{dt}\D(V(t))\leq2\kappa_1R_V^c\D({\mathcal B}(t)),\quad \mbox{a.e. }t>0.
\]
\end{proof}
Similar to the previous subsection, we first rearrange the terms in  $\eqref{B-10}_2$ as follows.
\begin{align}
\begin{aligned} \label{C-13-2}
\frac{dv_i}{dt} &= -\frac{\beta_i}{N} \Big[  \Big( \sum_{j=1}^N \chi_{ij}\phi(\|x_i-x_j\|)  \Big) v_i  - \sum_{j=1}^N \chi_{ij}\phi(\|x_i-x_j\|) v_j  \Big] \\
  &+  \frac{1}{N} \sum_{j=1}^N \chi_{ij}\phi(\|x_i-x_j\|) (\beta_j - \beta_i) v_j.
\end{aligned}
\end{align}
In order to express \eqref{C-13-2} in matrix form, we define an $N\times N$ matrix $\tilde {L}(t)$ by
\[
\tilde L(t):=\tilde D(t)-\tilde A(t),
\]
where the matrices $\tilde A(t)=(\tilde a_{ij}(t))$ and $\tilde D(t)=\mbox{diag}(\tilde d_1(t),\cdots,\tilde d_N(t))$ are defined by the following relations:
\[
\tilde a_{ij}(t):=\chi_{ij}\phi(\|x_i(t)-x_j(t)\|)\quad\mbox{and}\quad \tilde d_i(t)=\sum_{j=1}^N \chi_{ij}\phi(\|x_i(t)-x_j(t)\|).
\]
On the other hand, recall that
\[
 \Gamma(t):=\mbox{diag}\big(\beta_1(t),\cdots,\beta_N(t)\big).
\]
Thus, $\eqref{C-13-2}$ can be rewritten as
\begin{equation}\label{C-14}
\frac{d}{dt}V(t)=-\frac{1}{N}\Gamma(t)\tilde L(t)V(t)+\frac{1}{N}\Lambda(t),
\end{equation}
where the $N\times d$ matrix $\Lambda(t):=(\lambda_i^k(t))_{1\leq i\leq N,1\leq k\leq d}$ is defined by
\[
\lambda_i^k(t):=\sum_{j=1}^N\chi_{ij}\phi(\|x_i(t)-x_j(t)\|)(\beta_j(t)-\beta_i(t))v_j^k(t),~~ i=1,\cdots,N,~k=1,\cdots,d,~~ t\geq0.
\]
We also define the $N\times N$ matrix $B(t)=(b_{ij}(t))$ by
\[
b_{ij}(t)=\chi_{ij}\phi(\|x_i(t)-x_j(t)\|)(\beta_j(t)-\beta_i(t)),\quad i,j=1,\cdots,N.
\]
Then we have 
\begin{equation} \label{C-14-0} 
\Lambda(t)=B(t)V(t).
\end{equation}

Next, we perform the same analysis as in Lemma \ref{L3.2}. Let ${\tilde \Phi}(t_2,t_1)$ be the state transition matrix associated with the homogeneous part of  \eqref{C-14}. Then, for any given $\delta > 0$, we derive the solution formula for $V$:
\begin{equation} \label{C-14-1}
V\big(m \delta\big)=\tilde\Phi\big(m\delta,(m-1)\delta\big)V((m-1)\delta)+ \frac{1}{N} \int_{(m-1)\delta}^{m\delta}\tilde\Phi(m\delta,s) \Lambda(s)ds,\quad m \in \bbn.
\end{equation}
In next lemma, we study properties of the matrix $\tilde\Phi\big(m\delta,(m-1)\delta\big)$. 
\begin{lemma} \label{L3.5}
Let $\{ (x_i,v_i,\beta_i)\}$ be a solution to \eqref{B-10} satisfying a priori condition \eqref{C-10}. Then the following assertions hold.
\begin{enumerate}
\item
The ergodicity coefficient $\mu \Big({\tilde \Phi}\big(m\delta,(m-1)\delta\big)\Big)$ satisfies
\[ \mu \Big({\tilde \Phi}\big(m\delta,(m-1)\delta\big)\Big) \geq   C_2 \phi(x^{\infty})^{\gamma_g},   \]
where
\[
C_2=C_2(\delta):=e^{-\kappa_1\beta_U^{in} \delta}\cdot\frac{1}{\gamma_g!}\Big(\delta\frac{\beta_L^{in}}{N}\Big)^{\gamma_g}.
\]

\item
The state transition matrix ${\tilde \Phi}\big(m\delta,(m-1)\delta\big)$ is stochastic.
\end{enumerate}
\end{lemma}
\begin{proof} 
\noindent (1)~Note that
\begin{equation}\label{C-15}
-\frac{1}{N}\Gamma(t)\tilde L(t)=\frac{1}{N}\Gamma(t)(\tilde A(t)-\tilde D(t))\geq\frac{\beta_L^{in}}{N} {\tilde A}^{\infty}-\kappa_1\beta_U^{in}I,
\end{equation}
where ${\tilde A}^{\infty} =({\tilde a}^{\infty}_{ij})$ is a nonnegative matrix whose entries are defined by
\[
{\tilde a}_{ij}^{\infty}:=\chi_{ij}\phi(x^{\infty}).
\]
Motivated by \eqref{C-15}, we again decompose the matrix $-\frac{1}{N}\Gamma(t)\tilde L(t)$ into a sum of two matrices:
\[
-\frac{1}{N}\Gamma(t)\tilde L(t)=\left(-\frac{1}{N}\Gamma(t)\tilde L(t)+\kappa_1\beta_U^{in}I\right)-\kappa_1\beta_U^{in}I.
\]
Again, we use \eqref{C-15} to obtain
\begin{equation}\label{C-16}
-\frac{1}{N}\Gamma(t)\tilde L(t)+\kappa_1 \beta_U^{in}I\geq \frac{\beta_L^{in}}{N} {\tilde A}^{\infty}\geq0.
\end{equation}
For any $0\leq t_1<t_2<\infty$, let $\tilde \Phi(t_2,t_1)$ and $\tilde \Psi(t_2,t_1)$ be the state-transition matrices of $-\frac{1}{N}\Gamma(t)\tilde L(t)$ and $-\frac{1}{N}\Gamma(t)\tilde L(t)+\kappa_1 \beta_U^{in}I$ on $[t_1,t_2]$, respectively. Then, it follows from Lemma \ref{L2.3} that we have 
\begin{equation} \label{C-17}
\tilde \Phi\left(t_2,t_1\right)=e^{-\kappa_1\beta_U^{in}(t_2-t_1)}\tilde \Psi\left(t_2,t_1\right).
\end{equation}
And from \eqref{C-16}, we can use the Peano-Baker series and obtain the following:
\begin{align} \label{C-18}
\begin{aligned}
\tilde \Psi\left(t_2,t_1\right)&=I+\sum_{n=1}^\infty \int_{t_1}^{t_2}\int_{t_1}^{\tau_1}\cdots\int_{t_1}^{\tau_{n-1}}\Big((-\frac{1}{N}\Gamma(\tau_1)\tilde L(\tau_1)+\kappa_1\beta_U^{in}I)\cdots \\
&\hspace{2cm}(-\frac{1}{N}\Gamma(\tau_n)\tilde L(\tau_n)+\kappa_1\beta_U^{in}I) \Big)d\tau_n\cdots d\tau_1\\
&\geq I+\sum_{n=1}^\infty \int_{t_1}^{t_2}\int_{t_1}^{\tau_1}\cdots\int_{t_1}^{\tau_{n-1}}\Big(\frac{\beta_L^{in}}{N}{\tilde A}^{\infty} \Big)^n d\tau_n\cdots d\tau_1\\
&= I+\sum_{n=1}^\infty\frac{1}{n!} (t_2-t_1)^n\Big(\frac{\beta_L^{in}}{N} {\tilde A}^{\infty} \Big)^n = \exp\left((t_2-t_1)\frac{\beta_L^{in}}{N} {\tilde A}^{\infty}\right).
\end{aligned}
\end{align}
Now, we fix $m\in\bbn$, and combine \eqref{C-17} and \eqref{C-18}, and put $t_1=(m-1)\delta$, $t_2=m\delta$ to obtain 
\begin{align}\label{C-18-1}
\begin{aligned}
\tilde \Phi\Big(m \delta,(m-1)\delta\Big)&\geq e^{-\kappa_1\beta_U^{in} \delta} \exp\Big[ \delta\frac{\beta_L^{in}}{N} {\tilde A}^{\infty} \Big] =e^{-\kappa_1 \beta_U^{in} \delta}\Big[I+\sum_{n=1}^\infty\frac{1}{n!}\Big( \delta\frac{\beta_L^{in}}{N}{\tilde A}^{\infty} \Big)^n\Big]\\
&\geq e^{-\kappa_1 \beta_U^{in}\delta}\cdot\frac{1}{\gamma_g!}\Big(\delta\frac{\beta_L^{in}}{N}\Big)^{\gamma_g} ({\tilde A}^\infty)^{\gamma_g}= C_2 ({\tilde A}^\infty)^{\gamma_g} \geq0.
\end{aligned}
\end{align}
Hence, we use \eqref{B-1-1} and Lemma \ref{L2.1} to get
\begin{equation}\label{C-19}
\mu\Big(\tilde \Phi\Big(m \delta,(m-1)\delta\Big)\Big)\geq C_2\mu(({\tilde A}^\infty)^{\gamma_g})\geq C_2 \phi(x^{\infty})^{\gamma_g}.
\end{equation}
\newline
(2)~Note that by \eqref{C-18-1} ${\tilde \Phi}\Big(m\delta,(m-1)\delta\Big)$ is nonnegative and the constant state $\xi(t):=[\xi_1(t),\cdots,\xi_N(t)]^\top\equiv [1,\cdots,1]^\top$ is a solution to the corresponding homogeneous part of \eqref{C-14}:
\[ [1,\cdots,1]^\top=\tilde\Phi(m \delta,(m-1)\delta)[1,\cdots,1]^\top, \]
which implies that ${\tilde \Phi} \Big(m\delta,(m-1)\delta\Big)$ is stochastic. 
\end{proof}
Next, we are ready to provide the exponential decay estimate of $\D(V)$ in the following proposition.
\begin{proposition} [Exponential decay of $\D(V(t))$] \label{P3.2}
Let $\{ (x_i,v_i,\beta_i)\}$ be a solution to \eqref{B-10} satisfying a priori condition \eqref{C-10}. Then, we have the exponential decay of $\D(V(t))$: For any given $\delta>0$, we have 
\begin{align*}
\begin{aligned}
\D(V(t))&\leq   \left(1-C_2 \phi(x^{\infty})^{\gamma_g}\right)^{\lfloor \frac{t}{\delta}\rfloor}\D(V(0))+2\delta\kappa_1R_V^c\D({\mathcal B}(0))\left(1-C_1\zeta(x^{\infty})^{\gamma_g}\right)^{\lfloor\frac{t}{\delta}\rfloor}\\
&+\sqrt{2}N\kappa_1R_V^c\D({\mathcal B}(0))\delta\left\lfloor \frac{t}{\delta}\right\rfloor\Big[\max\{1-C_1\zeta(x^{\infty})^{\gamma_g},1-C_2 \phi(x^{\infty})^{\gamma_g}\}\Big]^{\lfloor \frac{t}{\delta}\rfloor-1},
\end{aligned}
\end{align*}
where $C_1=C_1(\delta)$ and $C_2=C_2(\delta)$ are the constants defined in Lemmas \ref{L3.2} and \ref{L3.5}, respectively.
\end{proposition}
\begin{proof}
We combine \eqref{C-14-1}, Lemma \ref{L2.2} and \eqref{C-19} to obtain
\begin{align}\label{C-20}
\begin{aligned}
\D(V(m\delta))&\leq\left(1-\mu\big(\tilde\Phi(m\delta,(m-1)\delta)\big)\right)\D(V((m-1)\delta))+\frac{\sqrt{2}}{N}\left\|\int_{(m-1)\delta}^{m\delta}\tilde\Phi(m\delta,s)\Lambda(s)ds\right\|_F\\
&\leq\left(1-C_2 \phi(x^{\infty})^{\gamma_g}\right)\D(V((m-1)\delta))+\frac{\sqrt{2}}{N}\int_{(m-1)\delta}^{m\delta}\|\tilde\Phi(m\delta,s)\|_F\|\Lambda(s)\|_Fds\\
&\leq\left(1-C_2 \phi(x^{\infty})^{\gamma_g}\right)\D(V((m-1)\delta))+\frac{\sqrt{2}}{N}\int_{(m-1)\delta}^{m\delta}\sqrt{N}\|\Lambda(s)\|_Fds.
\end{aligned}
\end{align}
The last inequality followed from the fact that $\tilde\Phi(m\delta,s)$ is a stochastic matrix with $N$ rows. We also use Proposition  \ref{P3.1}, Lemma \ref{L3.3} and \eqref{C-14-0} to get
\begin{align}\label{C-21}
\begin{aligned}
\|\Lambda(s)\|_F&\leq \|B(s)\|_F\|V(s)\|_F\leq N\kappa_1\D({\mathcal B}(s))\cdot \sqrt{N}\max_{1\leq i\leq N}\|v_i(s)\|\\
&\leq N\kappa_1\left(1-C_1\zeta(x^{\infty})^{\gamma_g}\right)^{\lfloor \frac{s}{\delta}\rfloor}\D({\mathcal B}(0))\cdot \sqrt{N}R_V^c\\
&\leq N\sqrt{N}\kappa_1R_V^c\D({\mathcal B}(0))\left(1-C_1\zeta(x^{\infty})^{\gamma_g}\right)^{ m-1},\quad s\in[(m-1)\delta,m\delta].
\end{aligned}
\end{align}
Finally, we combine \eqref{C-20} and \eqref{C-21} and get the following relation: for $ m\in\bbn$,
\begin{align}\label{C-22}
\begin{aligned}
&\D(V(m \delta)) \\
& \hspace{0.2cm} \leq \left(1-C_2 \phi(x^{\infty})^{\gamma_g}\right)\D(V((m-1)\delta)) +\sqrt{2}N\kappa_1R_V^c\D({\mathcal B}(0))\delta \left(1-C_1\zeta(x^{\infty})^{\gamma_g}\right)^{ m-1}\\
& \hspace{0.2cm} =:\left(1-C_2 \phi(x^{\infty})^{\gamma_g}\right)\D(V((m-1)\delta))+C_3 \left(1-C_1\zeta(x^{\infty})^{\gamma_g}\right)^{ m-1}.
\end{aligned}
\end{align}
We divide both sides of \eqref{C-22} by $\left(1-C_2 \phi(x^{\infty})^{\gamma_g}\right)^m$ to obtain 
\[
\frac{\D(V(m \delta))}{\left(1-C_2 \phi(x^{\infty})^{\gamma_g}\right)^m}\leq \frac{\D(V((m-1)\delta))}{\left(1-C_2 \phi(x^{\infty})^{\gamma_g}\right)^{m-1}}+\frac{C_3}{1-C_2 \phi(x^{\infty})^{\gamma_g}} \Bigg[\frac{1-C_1\zeta(x^{\infty})^{\gamma_g}}{1-C_2 \phi(x^{\infty})^{\gamma_g}}\Bigg]^{m-1},\quad m\in\bbn.
\]
This and inductive arguments yield
\[
\frac{\D(V(m \delta))}{\left(1-C_2 \phi(x^{\infty})^{\gamma_g}\right)^m}\leq\D(V(0))+\frac{C_3}{1-C_2 \phi(x^{\infty})^{\gamma_g}}\sum_{n=0}^{m-1}\Bigg[\frac{1-C_1\zeta(x^{\infty})^{\gamma_g}}{1-C_2 \phi(x^{\infty})^{\gamma_g}}\Bigg]^n,\quad m\in\bbn.
\]
Thus for any $m\in\bbn$ we have
\begin{align}\label{C-23}
\begin{aligned}
\D(V(m \delta)) &\leq \left(1-C_2 \phi(x^{\infty})^{\gamma_g}\right)^m\D(V(0)) \\
&+C_3\sum_{n=0}^{m-1}[1-C_1\zeta(x^{\infty})^{\gamma_g}]^n\left[1-C_2 \phi(x^{\infty})^{\gamma_g}\right]^{m-n-1}\\
&\leq \left(1-C_2 \phi(x^{\infty})^{\gamma_g}\right)^m\D(V(0)) \\
&+C_3m\Big[\max\{1-C_1\zeta(x^{\infty})^{\gamma_g},1-C_2 \phi(x^{\infty})^{\gamma_g}\}\Big]^{m-1}.
\end{aligned}
\end{align}
For any real $t=\delta p\geq0$, we combine Lemma \ref{L3.4}, Proposition \ref{P3.1}, and \eqref{C-23} to get
\begin{align*}
\begin{aligned}
&\D(V(t))=\D\left(V(\delta p)\right)\leq \D\left(V(\delta\lfloor p\rfloor)\right)+\int_{\delta\lfloor p\rfloor}^{\delta p}2\kappa_1R_V^c\D(\mathcal B(s))ds\\
&\hspace{0.5cm}\leq \D\left(V(\delta\lfloor p\rfloor)\right)+\int_{\delta\lfloor p\rfloor}^{\delta p}2\kappa_1R_V^c\left(1-C_1\zeta(x^\infty)^{\gamma_g}\right)^{\lfloor\frac{ s}{\delta}\rfloor}\D(\mathcal B(0))ds \\
&\hspace{0.5cm}\leq \D\left(V(\delta\lfloor p\rfloor)\right)+2\delta\kappa_1R_V^c\D(\mathcal B(0))\left(1-C_1\zeta(x^\infty)^{\gamma_g}\right)^{\lfloor p\rfloor} \\
&\hspace{0.5cm}\leq  \left(1-C_2 \phi(x^{\infty})^{\gamma_g}\right)^{\lfloor p\rfloor}\D(V(0))+C_3\lfloor p\rfloor\Big[\max\{1-C_1\zeta(x^{\infty})^{\gamma_g},1-C_2 \phi(x^{\infty})^{\gamma_g}\}\Big]^{\lfloor p\rfloor-1}\\
&\hspace{0.5cm}+2\delta\kappa_1R_V^c\D({\mathcal B}(0))\left(1-C_1\zeta(x^{\infty})^{\gamma_g}\right)^{\lfloor p\rfloor} \\
&\hspace{0.5cm}=  \left(1-C_2 \phi(x^{\infty})^{\gamma_g}\right)^{\lfloor \frac{t}{\delta}\rfloor}\D(V(0))+C_3\left\lfloor \frac{t}{\delta}\right\rfloor\Big[\max\{1-C_1\zeta(x^{\infty})^{\gamma_g},1-C_2 \phi(x^{\infty})^{\gamma_g}\}\Big]^{\lfloor \frac{t}{\delta}\rfloor-1}\\
&\hspace{0.5cm}+2\delta\kappa_1R_V^c\D({\mathcal B}(0))\left(1-C_1\zeta(x^{\infty})^{\gamma_g}\right)^{\lfloor\frac{t}{\delta}\rfloor}\\
&\hspace{0.5cm}=\left(1-C_2 \phi(x^\infty)^{\gamma_g}\right)^{\lfloor \frac{t}{\delta}\rfloor}\D(V(0))+2\delta\kappa_1R_V^c\D({\mathcal B}(0))\left(1-C_1\zeta(x^{\infty})^{\gamma_g}\right)^{\lfloor\frac{t}{\delta}\rfloor}\\
&\hspace{0.5cm}+\sqrt{2}N\kappa_1R_V^c\D({\mathcal B}(0))\delta\left\lfloor \frac{t}{\delta}\right\rfloor\Big[\max\{1-C_1\zeta(x^{\infty})^{\gamma_g},1-C_2 \phi(x^{\infty})^{\gamma_g}\}\Big]^{\lfloor \frac{t}{\delta}\rfloor-1}.
\end{aligned}
\end{align*}
\end{proof}

\subsection{Emergence of mono-cluster flocking}
In this subsection, we present a mono-cluster flocking estimate. Note that in Proposition \ref{P3.1} and Proposition \ref{P3.2}, we have alignment estimates for temperatures and velocities under the following a priori condition:
\[
 \sup_{0\leq t<\infty}\D(X(t))\leq x^{\infty} < \infty. 
\]
In the sequel, we will look for sufficient condition to guarantee the above a priori condition in terms of initial data and system parameters. Roughly speaking, our sufficient conditions can be stated as follows. If the initial position, velocity, and temperature of the particles are close enough, i.e., their corresponding diameters are sufficiently small, then spatial diameter will stay as bounded, hence temperature and velocity alignments emerge exponentially fast. For positive constants $x^{\infty}>0$ and $\delta>0$, we recall constants defined before:
\begin{align*}
\begin{aligned}
C_1 &=C_1(\delta):=e^{-\kappa_2(\beta_U^{in})^2\delta}\cdot\frac{1}{\gamma_g!}\Big(\delta\frac{(\beta_L^{in})^2}{N}\Big)^{\gamma_g}, \\
C_2 &=C_2(\delta):=e^{-\kappa_1\beta_U^{in} \delta}\cdot\frac{1}{\gamma_g!}\Big(\delta\frac{\beta_L^{in}}{N}\Big)^{\gamma_g}, \\
R_V^c &=R_V^c(x^{\infty},\delta):=\frac{1}{\beta_L^{in}}R_u(0)\exp\left(\frac{\kappa_2\delta\beta_U^{in}\D({\mathcal B}(0))}{C_1\zeta(x^{\infty})^{\gamma_g}}\right).
\end{aligned}
\end{align*}

\begin{theorem}\label{T3.1}
Suppose that for a given positive constants $x^{\infty}>0$ and $\delta>0$, the initial data $(X^{in}, V^{in}, {\mathcal B}^{in})$ satisfy the following relation:
\begin{align}
\begin{aligned} \label{C-30}
& \D(X(0))+\frac{\D(V(0))\delta}{C_2 \phi(x^{\infty})^{\gamma_g}} \\
& \hspace{1cm} +\frac{\sqrt{2}N\kappa_1R_V^c\D({\mathcal B}(0))\delta^2}{\big[\min\{C_1\zeta(x^{\infty})^{\gamma_g}, C_2 \phi(x^{\infty})^{\gamma_g}\}\big]^2} +\frac{2\kappa_1R_V^c\D({\mathcal B}(0))\delta^2 }{C_1\zeta(x^{\infty})^{\gamma_g}}\leq x^{\infty}.
\end{aligned}
\end{align}
Then, we have 
\begin{eqnarray*}
&& (i)~\sup_{0\leq t<\infty}\D(X(t))\leq x^{\infty}, \quad \D({\mathcal B}(t))\leq\left(1-C_1\zeta(x^{\infty})^{\gamma_g}\right)^{\lfloor \frac{t}{\delta}\rfloor}\D({\mathcal B}(0)), \\
&& (ii)~\D(V(t))\leq   \left(1-C_2 \phi(x^{\infty})^{\gamma_g}\right)^{\lfloor \frac{t}{\delta}\rfloor}\D(V(0))+2\delta\kappa_1R_V^c\D({\mathcal B}(0))\left(1-C_1\zeta(x^{\infty})^{\gamma_g}\right)^{\lfloor\frac{t}{\delta}\rfloor}\\
&&\hspace{2.3cm}+\sqrt{2}N\kappa_1R_V^c\D({\mathcal B}(0))\delta\Big\lfloor \frac{t}{\delta}\Big\rfloor\Big[\max\{1-C_1\zeta(x^{\infty})^{\gamma_g},1-C_2 \phi(x^{\infty})^{\gamma_g}\}\Big]^{\lfloor \frac{t}{\delta}\rfloor-1}.
\end{eqnarray*}
\end{theorem}
\begin{proof} We will use continuity argument. For this, we define the set $S$ as follows:
\[
S:=\Big\{T>0:\quad \D(X(t))<x^{\infty},\quad t\in[0,T]\Big\}.
\]
Then, by \eqref{C-30}, the set $S$ is nonempty. Now, we claim that $\sup S=\infty$. Suppose not, i.e. $T^*:=\sup S<\infty$. Then,  we have
\begin{equation*}
\D(X(T^*))=x^{\infty}.
\end{equation*}
It follows from $\eqref{B-10}_1$ that we have
\begin{align*}
\begin{aligned}
&\|x_i(T^*)-x_j(T^*)\|\leq\|x_i(0)-x_j(0)\|+\int_0^{T^*} \|v_i(s)-v_j(s)\|ds\\
& \leq \D(X(0))+\int_0^{T^*} \D(V(s))ds\\
&<\D(X(0))+\D(V(0))\int_0^{\infty}\left(1-C_2 \phi(x^{\infty})^{\gamma_g}\right)^{\lfloor \frac{s}{\delta}\rfloor}ds\\
&+2\delta\kappa_1R_V^c\D({\mathcal B}(0))\int_0^{\infty}\left(1-C_1\zeta(x^{\infty})^{\gamma_g}\right)^{\lfloor\frac{s}{\delta}\rfloor}ds\\
&+\sqrt{2}N\kappa_1R_V^c\D({\mathcal B}(0))\delta\int_0^{\infty}\left\lfloor \frac{s}{\delta}\right\rfloor\left[\max\{1-C_1\zeta(x^{\infty})^{\gamma_g},1-C_2 \phi(x^{\infty})^{\gamma_g}\}\right]^{\lfloor \frac{s}{\delta}\rfloor-1}ds\\
&=\D(X(0))+\D(V(0))\delta\sum_{n=0}^{\infty}\left(1-C_2 \phi(x^{\infty})^{\gamma_g}\right)^n \\
&+2\kappa_1R_V^c\D({\mathcal B}(0))\delta^2\sum_{n=0}^{\infty}\left(1-C_1\zeta(x^{\infty})^{\gamma_g}\right)^n\\
&+\sqrt{2}N\kappa_1R_V^c\D({\mathcal B}(0))\delta^2\sum_{n=1}^{\infty}n\left[\max\{1-C_1\zeta(x^{\infty})^{\gamma_g},1-C_2 \phi(x^{\infty})^{\gamma_g}\}\right]^{n-1}\\
&=\D(X(0))+\frac{\D(V(0))\delta}{C_2 \phi(x^{\infty})^{\gamma_g}}+\frac{\sqrt{2}N\kappa_1R_V^c\D({\mathcal B}(0))\delta^2}{\big[\min\{C_1\zeta(x^{\infty})^{\gamma_g},C_2 \phi(x^{\infty})^{\gamma_g}\}\big]^2}+\frac{2\kappa_1R_V^c\D({\mathcal B}(0))\delta^2 }{C_1\zeta(x^{\infty})^{\gamma_g}}\\
& \leq x^{\infty} .
\end{aligned}
\end{align*}
This implies $T^*\in S$, which is a contradiction. Therefore we have $\sup S=\infty$, i.e. $(i)$ holds. $(ii)$ and $(iii)$ follow from $(i)$ by Propositions \ref{P3.1} and \ref{P3.2}.
\end{proof}



\section{Emergent dynamics of the discrete model}\label{sec:4}
\setcounter{equation}{0}
In this section, we present an asymptotic flocking estimate for the discrete model \eqref{B-11}. We perform our flocking estimate in the following three steps, which are mostly parallel to those of the continuous model.
\subsection{A priori temperature alignment}
 In this subsection, we will derive a priori asymptotic alignment in temperature under the a priori assumption on the uniform boundedness of spatial diameter. For the convenience of presentation, we deal with the system \eqref{B-11}. In the sequel, the order of presentation will be exactly parallel to that of the continuous model.
\begin{lemma}[Boundedness of temperatures]\label{L4.1}
Let $\{(x_i[t],v_i[t],\beta_i[t]) \}$, be a solution to the system \eqref{B-11} with initial data $\{ (x_i^{in},v_i^{in},\beta_i^{in})\}$. Suppose that the time-step satisfies
\begin{equation} \label{D-0-0}
0<h\leq \frac{1}{\kappa_2 (\beta_U^{in})^2}.
\end{equation}
Then, the following assertions hold:
\begin{eqnarray*}
&& (i)~\beta^{in}_L\leq \beta_i[t] \leq \beta^{in}_U,\quad i=1,\cdots,N,\quad t\in\bbn\cup\{0\},\\
&& (ii)~\D({\mathcal B}[t]) \mbox{ is monotone decreasing}.
\end{eqnarray*}
\end{lemma}
\begin{proof}
We define the maximal and minimal values for coldness as
\[
\beta_M[t]:=\max_{1\leq i\leq N}\beta_i[t],\quad \beta_m[t]:=\min_{1\leq i\leq N}\beta_i[t],\quad t\in\bbn\cup\{0\}.
\]
For each $t\in\bbn\cup\{0\}$, we choose extremal indices $1\leq m_t, ~M_t\leq N$ satisfying
\[ \beta_m[t]=\beta_{m_t}[t] \quad \mbox{and} \quad \beta_M[t]=\beta_{M_t}[t]. \]
We claim the following relation:
\begin{equation}\label{D-0}
\beta^{in}_L\leq \beta_m[t-1]\leq \beta_m[t]\leq \beta_M[t]\leq \beta_M[t-1]\leq \beta^{in}_U,\quad t\in\bbn,
\end{equation}
{\it Proof of claim \eqref{D-0}}: We will use the proof by induction. \newline

\noindent $\bullet$~(Initial step): The base case $t=1$ can be shown in almost the same way as in the following inductive step.  \newline

\noindent $\bullet$~(Inductive step): Suppose that the relation \eqref{D-0} holds for $t \geq 1$. Then for $t+1$, we have
\begin{align*}
\begin{aligned}
&\frac{1}{\beta_m[t+1]}-\frac{1}{\beta_m[t]}=\frac{1}{\beta_{m_{t+1}}[t+1]}-\frac{1}{\beta_{m_t}[t]}\\
&\hspace{0.5cm}=\frac{1}{\beta_{m_{t+1}}[t]}+\frac{h}{N}\sum_{j=1}^N \chi_{{m_{t+1}}j}\zeta(\|x_{m_{t+1}}[t]-x_j[t]\|)\Big(\beta_{m_{t+1}}[t]-\beta_j[t]\Big)-\frac{1}{\beta_{m_t}[t]}\\
&\hspace{0.5cm}\leq \frac{1}{\beta_{m_{t+1}}[t]}+\frac{h}{N}\sum_{j=1}^N \chi_{{m_{t+1}}j}\zeta(\|x_{m_{t+1}}[t]-x_j[t]\|)\Big(\beta_{m_{t+1}}[t]-\beta_{m_t}[t]\Big)-\frac{1}{\beta_{m_t}[t]}\\
&\hspace{0.5cm}=\Big(\beta_{m_{t+1}}[t]-\beta_{m_t}[t]\Big)\bigg(-\frac{1}{\beta_{m_{t+1}}[t]\beta_{m_t}[t]}+\frac{h}{N}\sum_{j=1}^N \chi_{{m_{t+1}}j}\zeta(\|x_{m_{t+1}}[t]-x_j[t]\|)\bigg)\\
&\hspace{0.5cm}\leq\Big(\beta_{m_{t+1}}[t]-\beta_{m_t}[t]\Big)\bigg(-\frac{1}{(\beta_U^{in})^2}+h\kappa_2\bigg)\leq0.
\end{aligned}
\end{align*}
The second and the last inequalities followed from the fact that
\[
\beta_i[t]-\beta_{m_t}[t]\geq0,~~i=1,\cdots,N~~\Rightarrow~~ \beta_{m_{t+1}}[t]-\beta_{m_t}[t]\geq0.
\]
Hence we have
\begin{equation} \label{D-1}
\beta_m[t+1] \geq \beta_m[t].
\end{equation}
On the other hand, we have
\begin{align*}
\begin{aligned}
&\frac{1}{\beta_M[t+1]}-\frac{1}{\beta_M[t]}=\frac{1}{\beta_{M_{t+1}}[t+1]}-\frac{1}{\beta_{M_t}[t]}\\
&\hspace{0.5cm}=\frac{1}{\beta_{M_{t+1}}[t]}+\frac{h}{N}\sum_{j=1}^N \chi_{{M_{t+1}}j}\zeta(\|x_{M_{t+1}}[t]-x_j[t]\|)\Big(\beta_{M_{t+1}}[t]-\beta_j[t]\Big)-\frac{1}{\beta_{M_t}[t]}\\
&\hspace{0.5cm}\geq \frac{1}{\beta_{M_{t+1}}[t]}+\frac{h}{N}\sum_{j=1}^N \chi_{{M_{t+1}}j}\zeta(\|x_{M_{t+1}}[t]-x_j[t]\|)\Big(\beta_{M_{t+1}}[t]-\beta_{M_t}[t]\Big)-\frac{1}{\beta_{M_t}[t]}\\
&\hspace{0.5cm}=\Big(\beta_{M_{t+1}}[t]-\beta_{M_t}[t]\Big)\bigg(-\frac{1}{\beta_{M_{t+1}}[t]\beta_{M_t}[t]}+\frac{h}{N}\sum_{j=1}^N \chi_{{M_{t+1}}j}\zeta(\|x_{M_{t+1}}[t]-x_j[t]\|)\bigg)\\
&\hspace{0.5cm}\geq\Big(\beta_{M_{t+1}}[t]-\beta_{M_t}[t]\Big)\bigg(-\frac{1}{(\beta^{in}_U)^2}+h \kappa_2\bigg)\geq0.
\end{aligned}
\end{align*}
The second and the last inequalities followed from the fact that
\[
\beta_i[t]-\beta_{M_t}[t]\leq0,~~i=1,\cdots,N~~\Rightarrow~~ \beta_{M_{t+1}}[t]-\beta_{M_t}[t]\leq0.
\]

Therefore, we have
\begin{equation} \label{D-2}
\beta_M[t+1] \leq \beta_M[t].
\end{equation}
Finally, we combine \eqref{D-1} and \eqref{D-2} to derive the estimate \eqref{D-0}. Hence \eqref{D-0} also holds for $t+1$, and the induction is complete.
\end{proof}

Next, note that $\eqref{B-11}_3$ is equivalent to the following relation:
\begin{equation} \label{D-3}
\beta_i[t+1] =\beta_i[t]+\frac{h}{N}\beta_i[t]\beta_i[t+1]\sum_{j=1}^N \chi_{ij}\zeta(\|x_i[t]-x_j[t]\|)\Big(\beta_j[t]-\beta_i[t]\Big),\quad t\in\bbn\cup\{0\}.
\end{equation}
In fact, we can rewrite $\eqref{D-3}$ in vector form. For this, we define an $N\times N$ matrix $L[t]$:
\[
L[t]:=D[t]-A[t],
\]
where the matrices $A[t]=(a_{ij}[t])$ and $D[t]=\mbox{diag}(d_1[t],\cdots,d_N[t])$ are defined by the following relations:
\[
a_{ij}[t]:=\chi_{ij}\zeta(\|x_i[t]-x_j[t]\|)\quad\mbox{and}\quad d_i[t]=\sum_{j=1}^N \chi_{ij}\zeta(\|x_i[t]-x_j[t]\|)
\]
We also define
\[
\Gamma[t]:=\mbox{diag}\big(\beta_1[t],\cdots,\beta_N[t]\big).
\]
Then we can rewrite $\eqref{D-3}$ as follows.
\begin{equation} \label{D-4}
{\mathcal B}[t+1]=\bigg(I-\frac{h}{N}\Gamma[t]\Gamma[t+1]L[t]\bigg) {\mathcal B}[t], \quad t\in\bbn\cup\{0\}.
\end{equation}
\begin{proposition}\label{P4.1}
Suppose that time-step and initial data satisfy \eqref{D-0-0}, and let $\{ (x_i,v_i,\beta_i)\}$ be a solution to system~\eqref{B-11} satisfying a priori condition:
\begin{equation}\label{D-5}
\sup_{t\in\bbn\cup\{0\}}\D(X[t])\leq x^{\infty} < \infty.
\end{equation}
Then, we have the exponential decay of $\D({\mathcal B}[t])$: for any given integer $n_0\geq\gamma_g$ we have
\[
\D({\mathcal B}[t])\leq\left(1-D_1\zeta(x^{\infty})^{\gamma_g}\right)^{\lfloor \frac{t}{n_0}\rfloor}\D({\mathcal B}[0]),\quad t\in\bbn\cup\{0\},
\]
where the positive constant $D_1$ is given as follows.
\[
D_1=D_1(n_0):=\binom{n_0}{\gamma_g}(1-h\kappa_2 (\beta^{in}_U)^2)^{n_0-\gamma_g}\Big(\frac{h (\beta^{in}_L)^2}{N}\Big)^{\gamma_g}.
\]
\end{proposition}
\begin{proof} First, note that
\[
-\frac{1}{N}\Gamma[t]\Gamma[t+1]L[t]=\frac{1}{N}\Gamma[t]\Gamma[t+1](A[t]-D[t])\geq\frac{(\beta^{in}_L)^2}{N}A^{\infty} -\kappa_2 (\beta^{in}_U)^2I,
\]
where $A^{\infty}=(a_{ij}^{\infty})$ is a nonnegative matrix defined by
\[
a_{ij}^\infty:=\chi_{ij}\zeta(x^{\infty}).
\]
Then, the terms inside the parenthesis of \eqref{D-4} can be estimated as follows. 
\begin{equation}\label{D-6}
I-\frac{h}{N}\Gamma[t]\Gamma[t+1]L[t]\geq (1-h\kappa_2 (\beta^{in}_U)^2)I+\frac{h (\beta^{in}_L)^2}{N}A^{\infty} \geq0,\quad t\in\bbn\cup\{0\}.
\end{equation}
For any $t_1,t_2\in\bbn\cup\{0\}$ with $t_2-t_1\geq\gamma_g$, define the matrix $\Phi[t_2,t_1]$ as follows:
\begin{align*}
\begin{aligned}
\Phi[t_2,t_1]&:=\Big(I-\frac{h}{N}\Gamma[t_2-1]\Gamma[t_2]L[t_2-1]\Big)\Big(I-\frac{h}{N}\Gamma[t_2-2]\Gamma[t_2-1]L[t_2-2]\Big)\cdots\\
&\times \Big(I-\frac{h}{N}\Gamma[t_1]\Gamma[t_1+1]L[t_1]\Big).
\end{aligned}
\end{align*}
Then it follows from \eqref{D-6} that
\begin{align} \label{D-8}
\begin{aligned}
\Phi\left[t_2,t_1\right]&\geq \Big((1-h\kappa_2(\beta^{in}_U)^2)I+\frac{h (\beta^{in}_L)^2}{N}A^{\infty} \Big)^{t_2-t_1}\\
&=\sum_{n=0}^{t_2-t_1} \binom{t_2-t_1}{n}(1-h\kappa_2 (\beta^{in}_U)^2)^{t_2-t_1-n}\Big(\frac{h(\beta_L^{in})^2}{N}A^{\infty} \Big)^{n}\\
&\geq\binom{t_2-t_1}{\gamma_g}(1-h\kappa_2 (\beta^{in}_U)^2)^{t_2-t_1-\gamma_g}
\Big(\frac{h (\beta^{in}_L)^2}{N}A^{\infty} \Big)^{\gamma_g}.
\end{aligned}
\end{align}
Now, we fix $m\in\bbn$, and put $t_1=(m-1)n_0$, $t_2=mn_0$ in \eqref{D-8} to obtain 
\[
\Phi\Big[mn_0,(m-1)n_0\Big]\geq \binom{n_0}{\gamma_g}(1-h\kappa_2 (\beta^{in}_U)^2)^{n_0-\gamma_g}\Big(\frac{h(\beta^{in}_L)^2}{N}A^{\infty} \Big)^{\gamma_g}= D_1
(A^{\infty})^{\gamma_g}\geq0.
\]
Therefore, we have
\begin{equation}\label{D-9}
\mu\Big(\Phi\Big[mn_0,(m-1)n_0\Big]\Big)\geq D_1\mu((A^{\infty})^{\gamma_g})\geq D_1 \zeta(x^{\infty})^{\gamma_g},
\end{equation}
where in the last inequality, we used \eqref{B-1-1} and Lemma \ref{L2.1}. \newline

The nonnegative matrix $\Phi\Big[mn_0,(m-1)n_0\Big]$ is actually stochastic, because we have the following for each $(m-1)n_0\leq t< mn_0$:
\[
[1,\cdots,1]^\top=\Big(I-\frac{h}{N}\Gamma[t]\Gamma[t+1]L[t]\Big)[1,\cdots,1]^\top.
\]
Now, it follows from the relation:
\[ {\mathcal B}\big[mn_0\big]=\Phi\big(mn_0,(m-1)n_0\big){\mathcal B}[(m-1)n_0] \]
that we can use Lemma \ref{L2.2} with $B=0$ and \eqref{D-9} to obtain 
\begin{align*}
\begin{aligned}
\D\left({\mathcal B}[mn_0]\right)&\leq\left(1-\mu\Big(\Phi[mn_0,(m-1)n_0]\Big)\right)\D({\mathcal B}[(m-1)n_0])\\
&\leq \left(1-D_1\zeta(x^{\infty})^{\gamma_g}\right)\D({\mathcal B}[(m-1)n_0]),\quad m\in\bbn.
\end{aligned}
\end{align*}
By induction, we have 
\[
\D\left({\mathcal B} \left[mn_0 \right]\right)\leq \left(1-D_1\zeta(x^{\infty})^{\gamma_g}\right)^m\D({\mathcal B}[0]),\quad m\in\bbn.
\]
So for any $t\in\bbn\cup\{0\}$, we have the following:
\[
\D({\mathcal B}[t])\leq \D\left({\mathcal B}\Big[n_0\Big\lfloor \frac{t}{n_0}\Big\rfloor\Big]\right)\leq \left(1-D_1\zeta(x^{\infty})^{\gamma_g}\right)^{\lfloor \frac{t}{n_0}\rfloor}\D({\mathcal B}[0]).
\]
The first inequality was due to Lemma \ref{L4.1}.
\end{proof}

\subsection{A priori velocity alignment} In this subsection, we will derive asymptotic velocity alignment under the a priori assumption on the uniform boundedness of spatial diameters. For the convenience of presentation, we set
\begin{align*}
\begin{aligned} 
& u_i[t]:=\frac{v_i[t]}{\theta_i[t]}=\beta_i[t]v_i[t],\quad i=1,\cdots,N,\quad \mbox{and} \\
& R_u[t]:=\max_{1\leq i\leq N}\|u_i[t]\|,\quad t\in\bbn\cup\{0\}.
\end{aligned}
\end{align*}
As a first step, we study the boundedness of velocities.
\begin{lemma}[Boundedness of velocities] \label{L4.2}
Suppose that the time-step and initial data satisfy
\[ 0<h\leq\min \Big \{\frac{1}{\kappa_2 (\beta^{in}_U)^2},\frac{\beta^{in}_L}{2\kappa_1 (\beta^{in}_U)^2} \Big \},     \]
and let $\{ (x_i,v_i,\beta_i)\}$ be a solution to system~\eqref{B-11} satisfying a priori condition \eqref{D-5}. Then, velocities of the particles are bounded: for any given integer $n_0\geq \gamma_g$ we have
\[ \|v_i[t]\|\leq \frac{1}{\beta^{in}_L} R_u[0]\exp\bigg[\frac{h n_0\kappa_2 \beta^{in}_U\D({\mathcal B}[0])}{D_1\zeta(x^{\infty})^{\gamma_g}}\bigg]=:R_V^d=R_V^d(x^\infty,n_0),\quad t\in\bbn\cup\{0\}, \]
where $D_1$ is the constant defined in Proposition \ref{P4.1}.
\end{lemma}
\begin{proof} Since the proof is rather lengthy, we leave its proof in \ref{App-A}.
\end{proof}
Our next job is to introduce an inequality for $\D(V)$ which will be used later.
\begin{lemma} \label{L4.3}
Suppose that the time-step and initial data satisfy
\[ 0<h\leq\min \Big \{\frac{1}{\kappa_2 (\beta^{in}_U)^2},\frac{\beta^{in}_L}{2\kappa_1 (\beta^{in}_U)^2} \Big \},     \]
and let $\{ (x_i,v_i,\beta_i)\}$ be a solution to system~\eqref{B-11} satisfying a priori condition \eqref{D-5}. Then, for any given integer $n_0\geq\gamma_g$ we have
\[
\D(V[t+1])\leq\D(V[t])+ 2h\kappa_1 R_V^d\D({\mathcal B}[t]),\quad t\in\bbn\cup\{0\},
\]
where $R_V^d=R_V^d(x^\infty,n_0)$ is the constant defined in Lemma \ref{L4.2}.
\end{lemma} 
\begin{proof} Since the proof is lengthy, we leave its proof in Appendix B.
\end{proof}

\begin{proposition}[Exponential decay of the velocity diameter]  \label{P4.2}
Suppose that the time-step and initial data satisfy
\[ 0<h\leq\min \Big \{\frac{1}{\kappa_2 (\beta^{in}_U)^2},\frac{\beta^{in}_L}{2\kappa_1 (\beta^{in}_U)^2} \Big \},     \]
and let $\{ (x_i,v_i,\beta_i)\}$ be a solution to system~\eqref{B-11} satisfying a priori condition \eqref{D-5}. Then we have the exponential decay of $\D(V[t])$: for any given integer $n_0\geq\gamma_g$ we have
\begin{align*}
\begin{aligned}
\D(V[t])&\leq   \left(1-D_2 \phi(x^{\infty})^{\gamma_g}\right)^{\lfloor\frac{t}{n_0}\rfloor}\D(V[0])+2hn_0\kappa_1R_V^d\left(1-D_1\zeta(x^{\infty})^{\gamma_g}\right)^{\lfloor\frac{t}{n_0}\rfloor}\D({\mathcal B}[0])\\
&\hspace{0.5cm}+\sqrt{2}hn_0N\kappa_1R_V^d\D({\mathcal B}[0])\Big\lfloor\frac{t}{n_0}\Big\rfloor\Big[\max\{1-D_1\zeta(x^{\infty})^{\gamma_g},1-D_2 \phi(x^{\infty})^{\gamma_g}\}\Big]^{\lfloor\frac{t}{n_0}\rfloor-1},
\end{aligned}
\end{align*}
where $D_1$ is the constant defined in Proposition \ref{P4.1}, and
\[
D_2=D_2(n_0):=\binom{n_0}{\gamma_g}(1-h\kappa_1 \beta^{in}_U)^{n_0-\gamma_g}\Big(\frac{h\beta^{in}_L}{N}\Big)^{\gamma_g}.
\]
\end{proposition}

\begin{proof} We leave its proof in Appendix C.
\end{proof}

\subsection{Emergence of mono-cluster flocking}
In this subsection, we derive a mono-cluster flocking estimate by verifying the a proiri assumption \eqref{D-5} by imposing some conditions on system parameters and initial data. More precisely, our second main result can be summarized as follows. We set
\begin{align*}
\begin{aligned}
D_1 &=D_1(n_0):=\binom{n_0}{\gamma_g}(1-h\kappa_2(\beta^{in}_U)^2)^{n_0-\gamma_g}\Big(\frac{h (\beta^{in}_L)^2}{N}\Big)^{\gamma_g}, \\
D_2 &=D_2(n_0):=\binom{n_0}{\gamma_g}(1-h\kappa_1\beta^{in}_U)^{n_0-\gamma_g}\Big(\frac{h\beta^{in}_L}{N}\Big)^{\gamma_g}, \\
R_V^d &=R_V^d(\alpha,n_0):=\frac{1}{\beta^{in}_L}R_u[0]\exp\bigg[\frac{h n_0\kappa_2\beta^{in}_U\D({\mathcal B}[0])}{D_1\zeta(x^{\infty})^{\gamma_g}}\bigg].
\end{aligned}
\end{align*}
\begin{theorem}\label{T4.1}
Let a real number $\alpha>0$ and an integer $n_0\geq\gamma_g$ be given, and suppose that the time-step and initial data satisfy
\begin{align}
\begin{aligned} \label{D-30}
& 0<h\leq\min \Big \{\frac{1}{\kappa_2 (\beta^{in}_U)^2},\frac{\beta^{in}_L}{2\kappa_1 (\beta^{in}_U)^2} \Big \}, \\
& \D(X[0])+\frac{hn_0\D(V[0])}{D_2 \phi(x^{\infty})^{\gamma_g}}+\frac{\sqrt{2}h^2n_0^2N\kappa_1R_V^d\D({\mathcal B}[0])}{\big[\min\{D_1\zeta(x^{\infty})^{\gamma_g},D_2 \phi(x^{\infty})^{\gamma_g}\}\big]^2} \\
& \hspace{5cm} +\frac{2h^2n_0^2\kappa_1R_V^d\D({\mathcal B}[0]) }{D_1\zeta(x^{\infty})^{\gamma_g}}\leq x^{\infty},
\end{aligned}
\end{align}
and let $\{ (x_i,v_i,\beta_i)\}$ be a solution to system~\eqref{B-11}.  Then we have
\begin{eqnarray*}
&& (i)~\sup_{t\in\bbn\cup\{0\}}\D(X[t])\leq x^{\infty}, \quad \D({\mathcal B}[t])\leq\left(1-D_1\zeta(x^{\infty})^{\gamma_g}\right)^{\lfloor \frac{t}{n_0}\rfloor}\D({\mathcal B}[0]), \\
&& (ii)~\D(V[t])\leq   \left(1-D_2 \phi(x^{\infty})^{\gamma_g}\right)^{\lfloor\frac{t}{n_0}\rfloor}\D(V[0])+2hn_0\kappa_1R_V^d\left(1-D_1\zeta(x^{\infty})^{\gamma_g}\right)^{\lfloor\frac{t}{n_0}\rfloor}\D({\mathcal B}[0])\\
&&\hspace{1cm}+\sqrt{2}hn_0N\kappa_1R_V^d\D({\mathcal B}[0])\Big\lfloor\frac{t}{n_0}\Big\rfloor\Big[\max\{1-D_1\zeta(x^{\infty})^{\gamma_g},1-D_2 \phi(x^{\infty})^{\gamma_g}\}\Big]^{\lfloor\frac{t}{n_0}\rfloor-1}.
\end{eqnarray*}
\end{theorem}
\begin{proof} We claim:
\[ \D(X[t])\leq x^{\infty} \quad \mbox{for}~ t\in\bbn\cup\{0\}. \]
 We will prove this by induction on $t$. \newline
 
\noindent $\bullet$~Initial step: For the case $t=0$, it is clear from \eqref{D-30}.  \newline

\noindent $\bullet$~Induction step: Suppose that the claim holds for $0\leq t\leq m$. Then, we have
\begin{align*}
\begin{aligned}
&\|x_i[m+1]-x_j[m+1]\| \\
& \hspace{1cm} \leq\|x_i[0]-x_j[0]\|+h\sum_{n=0}^{m} \|v_i[n]-v_j[n]\|\\
&\hspace{1cm}\leq \D(X[0])+h\sum_{n=0}^{m}\D(V[n])\\
&\hspace{1cm}\leq\D(X[0])+h\D(V[0])\sum_{n=0}^{\infty}\left(1-D_2 \phi(x^{\infty})^{\gamma_g}\right)^{\lfloor \frac{n}{n_0}\rfloor}\\
&\hspace{1cm}+2h^2n_0\kappa_1R_V^d\D({\mathcal B}[0])\sum_{n=0}^{\infty}\left(1-D_1\zeta(x^{\infty})^{\gamma_g}\right)^{\lfloor\frac{n}{n_0}\rfloor}\\
&\hspace{1cm}+\sqrt{2}h^2n_0N\kappa_1R_V^d\D({\mathcal B}[0])\sum_{n=0}^{\infty}\left\lfloor \frac{n}{n_0}\right\rfloor\left[\max\{1-D_1\zeta(x^{\infty})^{\gamma_g},1-D_2 \phi(x^{\infty})^{\gamma_g}\}\right]^{\lfloor \frac{n}{n_0}\rfloor-1}\\
&\hspace{1cm}=\D(X[0])+hn_0\D(V[0])\sum_{n=0}^{\infty}\left(1-D_2 \phi(x^{\infty})^{\gamma_g}\right)^{n}\\
&\hspace{1cm}+2h^2n_0^2\kappa_1R_V^d\D({\mathcal B}[0])\sum_{n=0}^{\infty}\left(1-D_1\zeta(x^{\infty})^{\gamma_g}\right)^{n}\\
&\hspace{1cm}+\sqrt{2}h^2n_0^2N\kappa_1R_V^d\D({\mathcal B}[0])\sum_{n=1}^{\infty}n\left[\max\{1-D_1\zeta(x^{\infty})^{\gamma_g},1-D_2 \phi(x^{\infty})^{\gamma_g}\}\right]^{n-1}\\
&\hspace{1cm}=\D(X[0])+\frac{hn_0\D(V[0])}{D_2 \phi(x^{\infty})^{\gamma_g}}+\frac{\sqrt{2}h^2n_0^2N\kappa_1R_V^d\D({\mathcal B}[0])}{\big[\min\{D_1\zeta(x^{\infty})^{\gamma_g},D_2 \phi(x^{\infty})^{\gamma_g}\}\big]^2}+\frac{2h^2n_0^2\kappa_1R_V^d\D({\mathcal B}[0]) }{D_1\zeta(x^{\infty})^{\gamma_g}}\\
&\hspace{1cm} \leq x^{\infty}.
\end{aligned}
\end{align*}
Therefore, the claim holds for $t=m+1$, and the induction is complete. So a priori condition $(i)$ does hold, and the alignment estimates $(ii)$ and $(iii)$ follow from Proposition \ref{P4.1} and Proposition \ref{P4.2} respectively. 
\end{proof}
\begin{remark}\label{R4.2}
Fix a real number $\delta>0$ and take $n_0=\lfloor \frac{\delta}{h}\rfloor$. Then as $h\to0$, the left-hand side of \eqref{D-30} approaches that of \eqref{C-30}.
\end{remark}

\section{Conclusion} \label{sec:5}
\setcounter{equation}{0}
In this paper, we presented a mono-cluster flocking estimate for a thermodynamic Cucker-Smale model. As aforementioned in Introduction, most flocking models in literature deal with mechanical models, i.e., position and momentum are macroscopic observables. Thus, internal structures of particles, e.g., spin, temperature, vibration, etc, are often ignored in the modeling. Recently, Ha and Ruggeri introduced thermodynamic particle models which are consistent with thermodynamics and the Cucker-Smale model for the isothermal case. They derived the generalized Cucker-Smale model with temperature from the gas mixture models. Thus, it inherits the entropy principle as gas mixture system does. In a previous series of works on the emergent dynamics on the TCS model, most flocking analysis has been done mostly for the complete networks. Thus, interaction between network structure and system dynamics are completely decoupled. In this work, we presented exponential flocking estimates for the continuous and discrete TCS model with small diffusion velocities. Our proposed frameworks are formulated in terms of system parameters and initial data. Of course, there are many issues which have not been addressed in this paper. For example, we have only dealt with mono-cluster flocking. However, as noticed in the Cucker-Smale model, depending on the intial data and nature of communication weight (short range or long range), we might have multi-cluster flockings. Moreover, we do not have a detailed information on the spatial structure of resulting asymptotic flocking states. We leave these interesting issues for a future work.

\appendix

\newpage

\section{Proof of Lemma \ref{L4.2}} \label{App-A}
\setcounter{equation}{0}
For the proof, it suffices to show the upper bound of $R_u[t]$:
\[
R_u[t]\leq R_u[0]\exp\bigg[\frac{h n_0\kappa_2 \beta^{in}_U\D({\mathcal B}[0])}{D_1\zeta(x^{\infty})^{\gamma_g}}\bigg],\quad t\in\bbn\cup\{0\}.
\]
For each $t\in\bbn\cup\{0\}$, we choose an extremal index $1\leq M_t\leq N$ satisfying the relation:
\[ \|u_{M_t}[t]\|=R_u[t]. \]
For each $i=1,\cdots,N$, we have
\begin{align}\label{D-11}
\begin{aligned}
&\|u_i[t+1]\|^2-\|u_i[t]\|^2=\beta_i[t+1]^2\|v_i[t+1]\|^2-\beta_i[t]^2\|v_i[t]\|^2\\
&\hspace{1cm}=\|v_i[t]\|^2\big(\beta_i[t+1]^2-\beta_i[t]^2\big) + \beta_i[t+1]^2\big(\|v_i[t+1]\|^2-\|v_i[t]\|^2\big)\\
&\hspace{1cm}=:\mathcal I_{21} +\mathcal I_{22},\quad t\in\bbn\cup\{0\}.
\end{aligned}
\end{align}
Below, we estimate the terms ${\mathcal I}_{2i},~i=1,2$ one by one. \newline

\noindent $\bullet$ (Estimate of $\mathcal{I}_{21}$) : By direct calculation, we have
\begin{align}\label{D-12}
\begin{aligned}
\mathcal I_{21} &=\|v_i[t]\|^2\big(\beta_i[t+1]+\beta_i[t]\big)\big(\beta_i[t+1]-\beta_i[t]\big)\\
&=\frac{h}{N}\|v_i[t]\|^2\big(\beta_i[t+1]+\beta_i[t]\big)\beta_i[t]\beta_i[t+1]\sum_{j=1}^N \chi_{ij}\zeta(\|x_i[t]-x_j[t]\|)\Big(\beta_j[t]-\beta_i[t]\Big)\\
&=\frac{h}{N}\bigg(1+\frac{\beta_i[t+1]}{\beta_i[t]}\bigg)\beta_i[t+1]\|u_i[t]\|^2\sum_{j=1}^N \chi_{ij}\zeta(\|x_i[t]-x_j[t]\|)\Big(\beta_j[t]-\beta_i[t]\Big)\\
&\leq  \bigg(1+\frac{\beta_i[t+1]}{\beta_i[t]}\bigg)h\kappa_2\beta_U^{in}\D(\mathcal B[t])\|u_i[t]\|^2\\
&=  \bigg(2+\frac{h}{N}\beta_i[t+1]\sum_{j=1}^N \chi_{ij}\zeta(\|x_i[t]-x_j[t]\|)\big(\beta_j[t]-\beta_i[t]\big)\bigg)h\kappa_2\beta^{in}_U\D({\mathcal B}[t])\|u_i[t]\|^2\\
&\leq  \bigg(2+h \kappa_2 \beta^{in}_U \D\big({\mathcal B}[t]\big)\bigg)h\kappa_2 \beta^{in}_U \D({\mathcal B}[t])\|u_i[t]\|^2\\
\end{aligned}
\end{align}

\noindent $\bullet$ (Estimate of $\mathcal{I}_{22}$): We set 
\[
P:=\frac{h}{N}\sum_{j=1}^N \chi_{ij}\phi(\|x_i[t]-x_j[t]\|)\quad \mbox{and}\quad Q:=\frac{h}{N}\sum_{j=1}^N \chi_{ij}\phi(\|x_i[t]-x_j[t]\|)\|u_j[t]\|.
\]
Then, we use the estimate 
\begin{equation}\label{D-12-1}
P\leq h \kappa_1\leq \frac{1}{\beta^{in}_U}\leq\frac{1}{\beta_i[t]}
\end{equation}
to obtain
\begin{align*}
\begin{aligned}
\mathcal I_{22} &=\beta_i[t+1]^2\bigg(\Big\|v_i[t]+\frac{h}{N}\sum_{j=1}^N \chi_{ij}\phi(\|x_i[t]-x_j[t]\|)\Big(\beta_j[t]v_j[t]-\beta_i[t]v_i[t]\Big)\Big\|^2-\|v_i[t]\|^2\bigg)\\
&=\beta_i[t+1]^2\bigg(2\Big\langle v_i[t],\frac{h}{N}\sum_{j=1}^N \chi_{ij}\phi(\|x_i[t]-x_j[t]\|)\Big(\beta_j[t]v_j[t]-\beta_i[t]v_i[t]\Big)\Big\rangle\\
&+\Big\|\frac{h}{N}\sum_{j=1}^N \chi_{ij}\phi(\|x_i[t]-x_j[t]\|)\Big(\beta_j[t]v_j[t]-\beta_i[t]v_i[t]\Big)\Big\|^2\bigg)\\
&=\beta_i[t+1]^2\bigg(2\Big\langle \frac{u_i[t]}{\beta_i[t]},\frac{h}{N}\sum_{j=1}^N \chi_{ij}\phi(\|x_i[t]-x_j[t]\|)u_j[t]-Pu_i[t]\Big\rangle\\
&+\Big\|\frac{h}{N}\sum_{j=1}^N \chi_{ij}\phi(\|x_i[t]-x_j[t]\|)u_j[t]-Pu_i[t]\Big\|^2\bigg)\\
&=\beta_i[t+1]^2\bigg(\Big(P^2-\frac{2P}{\beta_i[t]}\Big)\|u_i[t]\|^2+\Big\|\frac{h}{N}\sum_{j=1}^N \chi_{ij}\phi(\|x_i[t]-x_j[t]\|)u_j[t]\Big\|^2\\
&+\Big(\frac{1}{\beta_i[t]}-P\Big)\frac{2h}{N}\sum_{j=1}^N \chi_{ij}\phi(\|x_i[t]-x_j[t]\|)\big\langle u_i[t],u_j[t]\big\rangle\bigg)\\
&\leq\beta_i[t+1]^2\bigg(\Big(P^2-\frac{2P}{\beta_i[t]}\Big)\|u_i[t]\|^2+\Big(\frac{h}{N}\sum_{j=1}^N \chi_{ij}\phi(\|x_i[t]-x_j[t]\|)\|u_j[t]\|\Big)^2\\
&+\Big(\frac{1}{\beta_i[t]}-P\Big)\frac{2h}{N}\sum_{j=1}^N \chi_{ij}\phi(\|x_i[t]-x_j[t]\|)\| u_i[t]\|\|u_j[t]\|\bigg)\\
&=\beta_i[t+1]^2\bigg(Q-P\|u_i[t]\|\bigg)\bigg(Q-\Big(P-\frac{2}{\beta_i[t]}\Big)\|u_i[t]\|\bigg)\\
&=: {\mathcal F}_1(Q).
\end{aligned}
\end{align*}
Note that ${\mathcal F}_1$ is a convex function in $Q$, and we have 
\[ 0\leq Q\leq P\|u_{M_t}[t]\|. \]
Thus, we have
\[
{\mathcal F}_1(Q)\leq\max\{{\mathcal F}_1(0),{\mathcal F}_1(P\|u_{M_t}[t]\|)\}.
\]
By \eqref{D-12-1}, we have
\[ {\mathcal F}_1(0)\leq 0\leq {\mathcal F}_1(P\|u_{M_t}[t]\|). \]
Hence, we have
\begin{align}\label{D-13}
\begin{aligned}
\mathcal I_{22} &\leq \beta_i[t+1]^2\bigg(P\|u_{M_t}[t]\|-P\|u_i[t]\|\bigg)\bigg(P\|u_{M_t}[t]\|-\Big(P-\frac{2}{\beta_i[t]}\Big)\|u_i[t]\|\bigg)\\
&= P\beta_i[t+1]^2\big(\|u_{M_t}[t]\|-\|u_i[t]\|\big)\bigg(\big(\|u_{M_t}[t]\|-\|u_i[t]\|\big)P+\frac{2}{\beta_i[t]}\|u_i[t]\|\bigg)\\
&\leq h\kappa_1 (\beta^{in}_U)^2\big(\|u_{M_t}[t]\|-\|u_i[t]\|\big)\bigg(\big(\|u_{M_t}[t]\|-\|u_i[t]\|\big)h\kappa_1+\frac{2}{(\beta^{in}_L)}\|u_i[t]\|\bigg)\\
&\leq h\kappa_1 (\beta^{in}_U)^2\big(\|u_{M_t}[t]\|-\|u_i[t]\|\big)\bigg(h\kappa_1\|u_{M_t}[t]\|
+\frac{2}{\beta^{in}_L}\|u_i[t]\|\bigg).
\end{aligned}
\end{align}
We combine \eqref{D-11}, \eqref{D-12}, and \eqref{D-13}  to obtain 
\begin{align}\label{D-13-0-1}
\begin{aligned}
&\|u_i[t+1]\|^2-\|u_i[t]\|^2 \leq \bigg(2+h \kappa_2 \beta^{in}_U\D\big({\mathcal B}[t]\big)\bigg)h\kappa_2 \beta^{in}_U \D({\mathcal B}[t])\|u_i[t]\|^2\\
&\hspace{1cm}+h\kappa_1 (\beta^{in}_U)^2\big(\|u_{M_t}[t]\|-\|u_i[t]\|\big)\bigg(h\kappa_1\|u_{M_t}[t]\|+\frac{2}{\beta^{in}_L}\|u_i[t]\|\bigg).
\end{aligned}
\end{align}
Now, we take $i=M_{t+1}$ in \eqref{D-13-0-1} to get
\begin{align*}
\begin{aligned}
&\|u_{M_{t+1}}[t+1]\|^2-\|u_{M_t}[t]\|^2 \\
&\hspace{1cm} =\|u_{M_{t+1}}[t]\|^2-\|u_{M_t}[t]\|^2+\|u_{M_{t+1}}[t+1]\|^2-\|u_{M_{t+1}}[t]\|^2\\
&\hspace{1cm}\leq\|u_{M_{t+1}}[t]\|^2-\|u_{M_t}[t]\|^2+\bigg(2+h \kappa_2\beta_U^{in}\D\big(\mathcal B[t]\big)\bigg)h\kappa_2 \beta^{in}_U\D(\mathcal B[t])\|u_{M_{t+1}}[t]\|^2\\
&\hspace{1cm}+h\kappa_1(\beta^{in}_U)^2\big(\|u_{M_t}[t]\|-\|u_{M_{t+1}}[t]\|\big)\bigg(h\kappa_1\|u_{M_t}[t]\|+\frac{2}{\beta^{in}_L}\|u_{M_{t+1}}[t]\|\bigg)\\
&\hspace{1cm}=\bigg(2+h \kappa_2\beta^{in}_U\D\big({\mathcal B}[t]\big)\bigg)h\kappa_2 \beta^{in}_U\D({\mathcal B}[t])\|u_{M_{t+1}}[t]\|^2\\
&\hspace{1cm}-\big(\|u_{M_t}[t]\|-\|u_{M_{t+1}}[t]\|\big)\bigg(\big(1-(h\kappa_1\beta^{in}_U)^2\big)\|u_{M_t}[t]\|+\big(1-\frac{2h\kappa_1 (\beta^{in}_U)^2}{\beta^{in}_L}\big)\|u_{M_{t+1}}[t]\|\bigg)\\
&\hspace{1cm}\leq \bigg(2+h \kappa_2\beta^{in}_U\D\big({\mathcal B}[t]\big)\bigg)h\kappa_2\beta^{in}_U\D({\mathcal B}[t])\|u_{M_{t+1}}[t]\|^2\\
&\hspace{1cm}\leq \bigg(2+h \kappa_2\beta^{in}_U\D\big({\mathcal B}[t]\big)\bigg)h\kappa_2\beta^{in}_U\D({\mathcal B}[t])
\|u_{M_t}[t]\|^2.
\end{aligned}
\end{align*}
This yields
\[
\|u_{M_{t+1}}[t+1]\|^2\leq\Big(1+h \kappa_2\beta^{in}_U\D\big({\mathcal B}[t]\big)\Big)^2\|u_{M_t}[t]\|^2.
\]
Hence, we have
\[
R_u[t+1]\leq\Big(1+h \kappa_2\beta^{in}_U\D\big({\mathcal B}[t]\big)\Big)R_u[t],\quad t\in\bbn\cup\{0\}.
\]
Now, we apply Proposition \ref{P4.1} to obtain 
\begin{align*}
\begin{aligned}
R_u[t]&\leq R_u[0]\prod_{n=0}^{t-1}\Big(1+h \kappa_2\beta^{in}_U\D\big({\mathcal B}[n]\big)\Big)\\
&=R_u[0]\exp\bigg[\sum_{n=0}^{t-1}\log\Big(1+h \kappa_2\beta^{in}_U\D\big({\mathcal B}[n]\big)\Big)\bigg]\\
&\leq R_u[0]\exp\bigg[\sum_{n=0}^{t-1}\Big(h \kappa_2\beta^{in}_U\D\big({\mathcal B}[n]\big)\Big)\bigg]\\
&\leq R_u[0]\exp\bigg[h \kappa_2\beta^{in}_U\D({\mathcal B}[0])\sum_{n=0}^{\infty}\left(1-D_1\zeta(x^{\infty})^{\gamma_g}\right)^{\lfloor \frac{n}{n_0}\rfloor}\bigg]\\
&= R_u[0]\exp\bigg[h \kappa_2\beta^{in}_U\D({\mathcal B}[0])\sum_{n=0}^{\infty}\sum_{k=nn_0}^{(n+1)n_0-1}\left(1-D_1\zeta(x^{\infty})^{\gamma_g}\right)^{\lfloor \frac{k}{n_0}\rfloor}\bigg]\\
&= R_u[0]\exp\bigg[h \kappa_2\beta^{in}_U\D({\mathcal B}[0])\sum_{n=0}^{\infty}n_0\left(1-D_1\zeta(x^{\infty})^{\gamma_g}\right)^{n}\bigg]\\
&= R_u[0]\exp\bigg[\frac{h n_0\kappa_2\beta^{in}_U\D({\mathcal B}[0])}{D_1\zeta(x^{\infty})^{\gamma_g}}\bigg].
\end{aligned}
\end{align*}
This completes the proof of Lemma \ref{L4.2}.

\section{Proof of Lemma \ref{L4.3}} \label{App-B}
\setcounter{equation}{0}

For each $t\in\bbn\cup\{0\}$, we choose extremal indices $1\leq i_t,~ j_t\leq N$ satisfying the relation:
\[ \D(V[t])=\big\|v_{i_t}[t]-v_{j_t}[t]\big\|.\] 
For each $i=1,\cdots,N$, we have
\begin{align}
\begin{aligned}\label{D-13-1}
&\|v_i[t+1]-v_j[t+1]\|^2-\|v_i[t]-v_j[t]\|^2 \\
&\hspace{0.2cm}=\bigg\|v_i[t]+\frac{h}{N}\sum_{k=1}^N \chi_{ik}\phi_{ik}[t]\Big(\beta_k[t]v_k[t]-\beta_i[t]v_i[t]\Big)\\
&\hspace{0.2cm}-v_j[t]-\frac{h}{N}\sum_{k=1}^N \chi_{jk}\phi_{jk}[t]\Big(\beta_k[t]v_k[t]-\beta_j[t]v_j[t]\Big)\bigg\|^2-\|v_i[t]-v_j[t]\|^2 \\
&\hspace{0.2cm}=\|(1-P)(v_i[t]-v_j[t])+X+Y\|^2-\|v_i[t]-v_j[t]\|^2,
\end{aligned}
\end{align}
where the quantities $P, X$ and $Y$ are defined as follows:
\begin{align*}
\begin{aligned}
& P:=\frac{h}{N}\sum_{k=1}^N \chi_{ik}\phi_{ik}[t]\beta_i[t]+\frac{h}{N}\sum_{k=1}^N \chi_{jk}\phi_{jk}[t]\beta_j[t], \\
& X:=\frac{h}{N}\sum_{k=1}^N \chi_{ik}\phi_{ik}[t]\Big(\beta_i[t]v_k[t]-\beta_i[t]v_j[t]\Big)-\frac{h}{N}\sum_{k=1}^N \chi_{jk}\phi_{jk}[t]\Big(\beta_j[t]v_k[t]-\beta_j[t]v_i[t]\Big), \\
& Y:=\frac{h}{N}\sum_{k=1}^N \chi_{ik}\phi_{ik}[t]\Big(\beta_k[t]v_k[t]-\beta_i[t]v_k[t]\Big)-\frac{h}{N}\sum_{k=1}^N \chi_{jk}\phi_{jk}[t]\Big(\beta_k[t]v_k[t]-\beta_j[t]v_k[t]\Big).
\end{aligned}
\end{align*}
We write $\phi_{ij}[t]:=\phi(\|x_i[t]-x_j[t]\|)$, $i,j=1,2,\cdots,N$ for convenience.\\
Next, we rewrite \eqref{D-13-1} as follows:
\begin{align}
\begin{aligned} \label{D-13-1-1}
&\|v_i[t+1]-v_j[t+1]\|^2-\|v_i[t]-v_j[t]\|^2 \\
& \hspace{1cm} =\Big(\|(1-P)(v_i[t]-v_j[t])+X\|^2-\|v_i[t]-v_j[t]\|^2\Big)\\
&\hspace{1cm}+\Big(\|(1-P)(v_i[t]-v_j[t])+X+Y\|^2-\|(1-P)(v_i[t]-v_j[t])+X\|^2\Big) \\
& \hspace{1cm} =:\mathcal I_{31} +\mathcal I_{32}.
\end{aligned}
\end{align}
Below, we estimate the terms ${\mathcal I}_{3i},~i=1,2$ one by one. \newline

\noindent $\bullet$ (Estimate of $\mathcal{I}_{31}$) :
Note that
\begin{equation}\label{D-13-2}
P\leq 2h\kappa_1\beta^{in}_U\leq 1.
\end{equation}
So we have
\begin{align*}
\begin{aligned}
\mathcal I_{31}&=(P^2-2P)\|v_i[t]-v_j[t]\|^2+2(1-P)\langle v_i[t]-v_j[t],X\rangle+\|X\|^2\\
&\leq(P^2-2P)\|v_i[t]-v_j[t]\|^2+2(1-P)\| v_i[t]-v_j[t]\|\|X\|+\|X\|^2\\
&=\Big(\|X\|-P\|v_i[t]-v_j[t]\|\Big)\Big(\|X\|+(2-P)\|v_i[t]-v_j[t]\|\Big)\\
&=: {\mathcal F}_2(\|X\|).
\end{aligned}
\end{align*}
Note that ${\mathcal F}_2$ is a convex function, and we have
\begin{align*}
\begin{aligned}
&0\leq\|X\|\leq\frac{h}{N}\sum_{k=1}^N \chi_{ik}\phi_{ik}[t]\beta_i[t]\big\|v_k[t]-v_j[t]\big\|+\frac{h}{N}\sum_{k=1}^N \chi_{jk}\phi_{jk}[t]\beta_j[t]\big\|v_k[t]-v_i[t]\big\|\\
&\hspace{1cm}\leq P\big\|v_{i_t}[t]-v_{j_t}[t]\big\|.
\end{aligned}
\end{align*}
Thus, we have
\[
{\mathcal F}_2(\|X\|)\leq\max \Big \{{\mathcal F}_2(0),~{\mathcal F}_2(P\big\|v_{i_t}[t]-v_{j_t}[t]\big\|) \Big \}.
\]
On the other hand, it follows from \eqref{D-13-2} that we have 
\[ {\mathcal F}_2(0)\leq 0\leq {\mathcal F}_2(P\big\|v_{i_t}[t]-v_{j_t}[t]\big\|). \]
Hence, we have
\begin{align}\label{D-13-3}
\begin{aligned}
\mathcal I_{31} &\leq {\mathcal F}_2(P\big\|v_{i_t}[t]-v_{j_t}[t]\big\|)\\
&=\Big(P\big\|v_{i_t}[t]-v_{j_t}[t]\big\|-P\big\|v_i[t]-v_j[t]\big\|\Big)\Big(P\big\|v_{i_t}[t]-v_{j_t}[t]\big\|+(2-P)\big\|v_i[t]-v_j[t]\big\|\Big)\\
&=\Big(\big\|v_{i_t}[t]-v_{j_t}[t]\big\|-\big\|v_i[t]-v_j[t]\big\|\Big)\Big(P^2\big\|v_{i_t}[t]-v_{j_t}[t]\big\|+(2P-P^2)\big\|v_i[t]-v_j[t]\big\|\Big)\\
&\leq\Big(\big\|v_{i_t}[t]-v_{j_t}[t]\big\|-\big\|v_i[t]-v_j[t]\big\|\Big)\Big(\big\|v_{i_t}[t]-v_{j_t}[t]\big\|+\big\|v_i[t]-v_j[t]\big\|\Big)\\
&=\big\|v_{i_t}[t]-v_{j_t}[t]\big\|^2-\big\|v_i[t]-v_j[t]\big\|^2\\
&=\D(V[t])^2-\big\|v_i[t]-v_j[t]\big\|^2.
\end{aligned}
\end{align}

\vspace{0.2cm}

\noindent $\bullet$ (Estimate of $\mathcal{I}_{32}$): In this case,  note that
\begin{align*}
\begin{aligned}
\big\|Y\big\|&\leq\frac{h}{N}\sum_{k=1}^N \chi_{ik}\phi_{ik}[t]\big|\beta_k[t]-\beta_i[t]\big|\big\|v_k[t]\big\|+\frac{h}{N}\sum_{k=1}^N \chi_{jk}\phi_{jk}[t]\big|\beta_k[t]-\beta_j[t]\big|\big\|v_k[t]\big\|\\
&\leq 2h\kappa_1 R_V^d\D({\mathcal B}[t])
\end{aligned}
\end{align*}
and
\[
\big\|(1-P)(v_i[t]-v_j[t])+X\big\|\leq \D(V[t]),
\]
by \eqref{D-13-3}. Hence, we have
\begin{align}\label{D-13-4}
\begin{aligned}
\mathcal I_{32}&=2\big\langle(1-P)(v_i[t]-v_j[t])+X,Y\big\rangle+\big\|Y\big\|^2\\
&\leq2\big\|(1-P)(v_i[t]-v_j[t])+X\big\|\big\|Y\big\|+\big\|Y\big\|^2\\
&\leq2\D(V[t])\cdot 2h\kappa_1 R_V^d\D({\mathcal B}[t])+\Big(2h\kappa_1 R_V^d\D({\mathcal B}[t])\Big)^2\\
&=\Big(\D(V[t])+ 2h\kappa_1 R_V^d\D({\mathcal B}[t])\Big)^2-\D(V[t])^2.
\end{aligned}
\end{align}
Now, it follows from \eqref{D-13-1-1}, \eqref{D-13-3} and \eqref{D-13-4}  that we have
\begin{align*}
\begin{aligned}
&\|v_i[t+1]-v_j[t+1]\|^2-\|v_i[t]-v_j[t]\|^2\\
&\hspace{1cm}\leq\Big(\D(V[t])+ 2h\kappa_1 R_V^d\D({\mathcal B}[t])\Big)^2-\big\|v_{i}[t]-v_{j}[t]\big\|^2,
\end{aligned}
\end{align*}
i.e.
\begin{equation}\label{D-13-5}
\|v_i[t+1]-v_j[t+1]\|^2\leq\Big(\D(V[t])+ 2h\kappa_1 R_V^d\D({\mathcal B}[t])\Big)^2.
\end{equation}
We take $(i,j)=(i_{t+1},j_{t+1})$ in \eqref{D-13-5} to obtain 
\[
\D(V[t+1])^2\leq\Big(\D(V[t])+ 2h\kappa_1 R_V^d\D({\mathcal B}[t])\Big)^2,
\]
i.e.
\[
\D(V[t+1])\leq\D(V[t])+ 2h\kappa_1 R_V^d\D({\mathcal B}[t]).
\]
This completes the proof of Lemma \ref{L4.3}.

\section{Proof of Proposition \ref{P4.2}} \label{App-C}
\setcounter{equation}{0}
We can rewrite $\eqref{B-11}_2$ in a more concise form.
We define an $N\times N$ matrix $\tilde {L}[t]$ by
\[
\tilde L[t]:=\tilde D[t]-\tilde A[t],
\]
where the matrices $\tilde A[t]=(\tilde a_{ij}[t])$ and $\tilde D[t]=\mbox{diag}(\tilde d_1[t],\cdots,\tilde d_N[t])$ are defined by the following relations:
\[
\tilde a_{ij}[t]:=\chi_{ij}\phi(\|x_i[t]-x_j[t]\|)\quad\mbox{and}\quad \tilde d_i[t]=\sum_{j=1}^N \chi_{ij}\phi(\|x_i[t]-x_j[t]\|).
\]
Recall that
\[
 \Gamma[t]:=\mbox{diag}\big(\beta_1[t],\cdots,\beta_N[t]\big).
\]
Then we can rewrite $\eqref{B-11}_2$ as
\begin{equation}\label{D-14}
V[t+1]=\Big[I-\frac{h}{N}\Gamma[t]\tilde L[t]\Big]V[t]+\frac{h}{N}\Lambda[t],
\end{equation}
where the $N\times d$ matrix $\Lambda[t]:=(\lambda_i^k[t])_{1\leq i\leq N,1\leq k\leq d}$ is defined by
\[
\lambda_i^k[t]:=\sum_{j=1}^N\chi_{ij}\phi(\|x_i[t]-x_j[t]\|)(\beta_j[t]-\beta_i[t])v_j^k[t],~~ i=1,\cdots,N,~~k=1,\cdots, d,~~ t\in\bbn\cup\{0\}.
\]
Lastly, we define the $N\times N$ matrix $B[t]=(b_{ij}[t])$ by
\[
b_{ij}[t]=\chi_{ij}\phi(\|x_i[t]-x_j[t]\|)(\beta_j[t]-\beta_i[t]),\quad i,j=1,\cdots,N.
\]
Then we have $\Lambda[t]=B[t]V[t]$.\newline

Note that
\[
-\frac{1}{N}\Gamma[t]\tilde L[t]=\frac{1}{N}\Gamma[t](\tilde A[t]-\tilde D[t])\geq\frac{\beta_L^{in}}{N} {\tilde A}^{\infty}-\kappa_1\beta^{in}_UI,
\]
where ${\tilde A}^{\infty}=({\tilde a}_{ij}^\infty)$ is a nonnegative matrix defined by
\[
{\tilde a}_{ij}^\infty:=\chi_{ij}\phi(x^{\infty}).
\]
Hence we have
\begin{equation}\label{D-16}
I-\frac{h}{N}\Gamma[t]\tilde L[t]\geq (1-h\kappa_1\beta^{in}_U)I+\frac{h\beta^{in}_L}{N} {\tilde A}^{\infty} \geq0,\quad t\in\bbn\cup\{0\}.
\end{equation}
For any $t_1,t_2\in\bbn\cup\{0\}$ with $t_2-t_1\geq\gamma_g$, define the matrix $\tilde\Phi[t_2,t_1]$ as follows:
\begin{align*}
\begin{aligned}
\tilde\Phi[t_2,t_1]&:=\Big(I-\frac{h}{N}\Gamma[t_2-1]\tilde L[t_2-1]\Big)\Big(I-\frac{h}{N}\Gamma[t_2-2]\tilde L[t_2-2]\Big)\cdots\Big(I-\frac{h}{N}\Gamma[t_1]\tilde L[t_1]\Big).
\end{aligned}
\end{align*}
It follows from \eqref{D-16} that
\begin{align} \label{D-18}
\begin{aligned}
\tilde\Phi\left[t_2,t_1\right]&\geq \Big((1-h\kappa_1\beta^{in}_U)I+\frac{h\beta^{in}_L}{N} {\tilde A}^{\infty}\Big)^{t_2-t_1}\\
&=\sum_{n=0}^{t_2-t_1} \binom{t_2-t_1}{n}(1-h\kappa_1\beta^{in}_U)^{t_2-t_1-n}\Big(\frac{h\beta^{in}_L}{N} {\tilde A}^{\infty}\Big)^{n}\\
&\geq\binom{t_2-t_1}{\gamma_g}(1-h\kappa_1 \beta^{in}_U)^{t_2-t_1-\gamma_g}\Big(\frac{h\beta^{in}_L}{N} {\tilde A}^{\infty}\Big)^{\gamma_g}.
\end{aligned}
\end{align}
Now, we  fix $m\in\bbn$ and put $t_1=(m-1)n_0$, $t_2=mn_0$ in \eqref{D-18} to obtain 
\[
\tilde\Phi\Big[mn_0,(m-1)n_0\Big]\geq \binom{n_0}{\gamma_g}(1-h\kappa_1\beta^{in}_U)^{n_0-\gamma_g}\Big(\frac{h\beta^{in}_L}{N}{\tilde A}^{\infty} \Big)^{\gamma_g}= D_2
({\tilde A}^{\infty})^{\gamma_g}\geq0.
\]
Then, we use Lemma \ref{L2.1} to derive
\begin{equation}\label{D-19}
\mu\Big(\tilde\Phi\Big[mn_0,(m-1)n_0\Big]\Big)\geq D_2\mu(({\tilde A}^{\infty})^{\gamma_g})\geq D_2 \phi(x^{\infty})^{\gamma_g}.
\end{equation}
Note that for each $(m-1)n_0\leq t< mn_0$:
\[
[1,\cdots,1]^\top=\Big(I-\frac{h}{N}\Gamma[t]\tilde L[t]\Big)[1,\cdots,1]^\top.
\]
Thus, the nonnegative matrix $\tilde \Phi\Big[mn_0,(m-1)n_0\Big]$ is actually stochastic. \newline

On the other hand, it follows from \eqref{D-14} that we have
\begin{align*}
\begin{aligned}
V\big[mn_0\big] &=\tilde\Phi\big[mn_0,(m-1)n_0\big]V[(m-1)n_0] \\
&+\sum_{n=0}^{n_0-1}\tilde\Phi\big[mn_0,(m-1)n_0+n+1\big]\Big(\frac{h}{N}\Lambda[(m-1)n_0+n]\Big),
\end{aligned}
\end{align*}
with the convention $\tilde\Phi\big[mn_0,mn_0\big]:=I$. \newline

Now, we use Lemma \ref{L2.2} and \eqref{D-19} to obtain 
\begin{align}\label{D-20}
\begin{aligned}
\D\left(V[mn_0]\right)&\leq\left(1-\mu\Big(\tilde\Phi[mn_0,(m-1)n_0]\Big)\right)\D(V[(m-1)n_0])\\
& \hspace{0.5cm} +\frac{\sqrt{2}h}{N}\bigg\|\sum_{n=0}^{n_0-1}\tilde\Phi\big[mn_0,(m-1)n_0+n+1\big]\Lambda[(m-1)n_0+n]\bigg\|_F\\
&\leq \left(1-D_2\phi(x^{\infty})^{\gamma_g}\right)\D(V[(m-1)n_0])\\
& \hspace{0.5cm} +\frac{\sqrt{2}h}{N}\sum_{n=0}^{n_0-1}\big\|\tilde\Phi\big[mn_0,(m-1)n_0+n+1\big]\big\|_F\big\|\Lambda[(m-1)n_0+n]\big\|_F\\
&\leq \left(1-D_2\phi(x^{\infty})^{\gamma_g}\right)\D(V[(m-1)n_0])
+\frac{\sqrt{2}h}{N}\sum_{n=0}^{n_0-1}\sqrt{N}\big\|\Lambda[(m-1)n_0+n]\big\|_F.\\
\end{aligned}
\end{align}
The last inequality follows from the fact that $\tilde\Phi\big[mn_0,(m-1)n_0+n+1\big]$ is a stochastic matrix with $N$ rows. We use Proposition  \ref{P4.1} and Lemma \ref{L4.2} to derive
\begin{align}\label{D-21}
\begin{aligned}
&\|\Lambda[(m-1)n_0+n]\|_F \\
&\hspace{0.5cm} \leq \|B[(m-1)n_0+n]\|_F\|V[(m-1)n_0+n]\|_F\\
&\hspace{0.5cm}\leq N\kappa_1\D(\mathcal B[(m-1)n_0+n])\cdot \sqrt{N}\max_{1\leq i\leq N}\|v_i[(m-1)n_0+n]\|\\
&\hspace{0.5cm} \leq N\kappa_1\left(1-D_1\zeta(x^{\infty})^{\gamma_g}\right)^{ m-1}\D({\mathcal B}[0])\cdot \sqrt{N}R_V^d,\quad 0\leq n\leq n_0-1.
\end{aligned}
\end{align}
Next, we combine \eqref{D-20} and \eqref{D-21} to get the following relation:
\begin{align}\label{D-22}
\begin{aligned}
\D(V[m n_0])&\leq \left(1-D_2 \phi(x^{\infty})^{\gamma_g}\right)\D(V[(m-1)n_0]) \\
&+\sqrt{2}hn_0N\kappa_1R_V^d\D({\mathcal B}[0]) \left(1-D_1\zeta(x^{\infty})^{\gamma_g}\right)^{ m-1}\\
&=:\left(1-D_2 \phi(x^{\infty})^{\gamma_g}\right)\D(V[(m-1)n_0])+D_3 \left(1-D_1\zeta(x^{\infty})^{\gamma_g}\right)^{ m-1},\quad m\in\bbn.
\end{aligned}
\end{align}
We divide both sides of \eqref{D-22} by $\left(1-D_2 \phi(x^{\infty})^{\gamma_g}\right)^m$ to obtain 
\[
\frac{\D(V[m n_0])}{\left(1-D_2 \phi(x^{\infty})^{\gamma_g}\right)^m}\leq \frac{\D(V[(m-1)n_0])}{\left(1-D_2 \phi(x^{\infty})^{\gamma_g}\right)^{m-1}}+\frac{D_3}{1-D_2 \phi(x^{\infty})^{\gamma_g}} \Bigg[\frac{1-D_1\zeta(x^{\infty})^{\gamma_g}}{1-D_2 \phi(x^{\infty})^{\gamma_g}}\Bigg]^{m-1},\quad m\in\bbn.
\]
This yields
\[
\frac{\D(V[m n_0])}{\left(1-D_2 \phi(x^{\infty})^{\gamma_g}\right)^m}\leq\D(V[0])+\frac{D_3}{1-D_2 \phi(x^{\infty})^{\gamma_g}}\sum_{n=0}^{m-1}\Bigg[\frac{1-D_1\zeta(x^{\infty})^{\gamma_g}}{1-D_2 \phi(x^{\infty})^{\gamma_g}}\Bigg]^n,\quad m\in\bbn.
\]
Thus for any $m\in\bbn$ we have
\begin{align}\label{D-23}
\begin{aligned}
&\D(V[mn_0]) \\
&\leq \left(1-D_2 \phi(x^{\infty})^{\gamma_g}\right)^m\D(V[0])+D_3\sum_{n=0}^{m-1}[1-D_1\zeta(x^{\infty})^{\gamma_g}]^n\left[1-D_2 \phi(x^{\infty})^{\gamma_g}\right]^{m-n-1}\\
&\leq \left(1-D_2 \phi(x^{\infty})^{\gamma_g}\right)^m\D(V[0])+D_3m\Big[\max\{1-D_1\zeta(x^{\infty})^{\gamma_g},1-D_2 \phi(x^{\infty})^{\gamma_g}\}\Big]^{m-1}.
\end{aligned}
\end{align}
So for any $t\in\bbn\cup\{0\}$, we combine Lemma \ref{L4.3}, Proposition \ref{P4.1}, and \eqref{D-23} to deduce 
\begin{align*}
\begin{aligned}
&\D(V[t])\leq \D\left(V\Big[n_0\Big\lfloor \frac{t}{n_0}\Big\rfloor\Big]\right)+\sum_{n=n_0\lfloor\frac{t}{n_0}\rfloor}^{t-1}2h\kappa_1R_V^d\D({\mathcal B}[n])\\
&\hspace{0.5cm}\leq \D\left(V\Big[n_0\Big\lfloor \frac{t}{n_0}\Big\rfloor\Big]\right)+2h\Big(t-n_0\Big\lfloor\frac{t}{n_0}\Big\rfloor\Big)\kappa_1R_V^d\left(1-D_1\zeta(x^{\infty})^{\gamma_g}\right)^{\lfloor\frac{t}{n_0}\rfloor}\D({\mathcal B}[0]) \\
&\hspace{0.5cm}\leq \D\left(V\Big[n_0\Big\lfloor \frac{t}{n_0}\Big\rfloor\Big]\right)+2hn_0\kappa_1R_V^d\left(1-D_1\zeta(x^{\infty})^{\gamma_g}\right)^{\lfloor\frac{t}{n_0}\rfloor}\D({\mathcal B}[0]) \\
&\hspace{0.5cm}\leq  \left(1-D_2 \phi(x^{\infty})^{\gamma_g}\right)^{\lfloor\frac{t}{n_0}\rfloor}\D(V[0])+D_3\Big\lfloor\frac{t}{n_0}\Big\rfloor\Big[\max\{1-D_1\zeta(x^{\infty})^{\gamma_g},1-D_2 \phi(x^{\infty})^{\gamma_g}\}\Big]^{\lfloor\frac{t}{n_0}\rfloor-1}\\
&\hspace{0.5cm}+2hn_0\kappa_1R_V^d\left(1-D_1\zeta(x^{\infty})^{\gamma_g}\right)^{\lfloor\frac{t}{n_0}\rfloor}\D({\mathcal B}[0]) \\
&\hspace{0.5cm}=\left(1-D_2 \phi(x^{\infty})^{\gamma_g}\right)^{\lfloor\frac{t}{n_0}\rfloor}\D(V[0])+2hn_0\kappa_1R_V^d\left(1-D_1\zeta(x^{\infty})^{\gamma_g}\right)^{\lfloor\frac{t}{n_0}\rfloor}\D({\mathcal B}[0])\\
&\hspace{0.5cm}+\sqrt{2}hn_0N\kappa_1R_V^d\D({\mathcal B}[0])\Big\lfloor\frac{t}{n_0}\Big\rfloor\Big[\max\{1-D_1\zeta(x^{\infty})^{\gamma_g},1-D_2 \phi(x^{\infty})^{\gamma_g}\}\Big]^{\lfloor\frac{t}{n_0}\rfloor-1}.
\end{aligned}
\end{align*}
This completes the proof.

\newpage

\noindent {\bf Acknowledgment}
The work of S.-Y. Ha is supported by the Samsung Science and Technology Foundation under Project Number SSTF-BA1401-03. The work of J.-G. Dong was supported in part by NSFC grant 11671109.

%

\end{document}